\newif\ifsubmit
\documentclass[a4paper]{article}
\usepackage{amsmath,amssymb,amsthm}
\usepackage{hyperref}
\usepackage{cleveref}
\usepackage[title]{appendix}
\usepackage{doi}
\usepackage[square,numbers]{natbib}
\hypersetup{pdfborder = 0 0 0, colorlinks=false,
 linkcolor=black,citecolor=black, filecolor=black, urlcolor=black, }
\newcommand{\email}{}

\usepackage{geometry}
\usepackage{bm}
\usepackage{commath}

\usepackage[backgroundcolor=yellow]{todonotes}

\renewcommand{\v}[1]{\bm{#1}}

\renewcommand{\Re}{\operatorname{Re}}
\renewcommand{\Im}{\operatorname{Im}}

\newcommand{\rebrac}[1]{\Re\left[#1\right]}
\newcommand{\conj}[1]{\bar{#1}}
\newcommand{\bdry}{{\partial\Omega}}
\newcommand{\nhat}{\hat{n}}

\newcommand{\pars}[1]{\left( #1 \right)}%
\newcommand{\braces}[1]{\left\{ #1 \right\}}%

\newcommand{\poly}{\operatorname{P}}
\newcommand{\Reps}{{\rho_\epsilon}}              
\newcommand{\be}{\begin{equation}}
\newcommand{\ee}{\end{equation}}
\newcommand{\cc}{\mathbb{C}}
\newcommand{\ft}{\tilde f}        
\newcommand{\bigO}{{\mathcal O}}
\newtheorem{thm}{Theorem}
\newtheorem{pro}[thm]{Proposition}
\newtheorem{rmk}[thm]{Remark}
\newtheorem{cnj}[thm]{Conjecture}
\newcommand{\stokeslet}{{\mathcal S}}
\newcommand{\doublet}{{\mathcal D}}
\newcommand{\epspan}{\epsilon}   

\title{Accurate quadrature of nearly singular line integrals in two and three dimensions by singularity swapping}

\author{Ludvig af Klinteberg\thanks{Department of Mathematics,
    Simon Fraser University, Burnaby, BC, Canada (\email{ludvigak@kth.se})}
  \and Alex H. Barnett\thanks{Center for Computational Mathematics, Flatiron Institute, NY, USA}
}

\begin{document}
\maketitle

\begin{abstract}
The method of Helsing and co-workers evaluates Laplace and related layer potentials generated by a panel (composite) quadrature on a curve,
efficiently and with high-order accuracy for arbitrarily close targets.
Since it exploits complex analysis, its use has been restricted to two dimensions (2D).
We first explain its loss of accuracy as panels become curved,
using a classical complex approximation result of Walsh
that can be interpreted as ``electrostatic shielding''
of a
Schwarz singularity.
We then introduce a variant that swaps
the target singularity for one at its
{\em complexified parameter preimage};
in the latter space the panel is flat, hence the
convergence rate can be much higher.
%
The preimage is found robustly by 
Newton iteration.
This idea also enables, for the first time, a near-singular quadrature
for potentials generated
by smooth curves in 3D, building on recurrences of Tornberg--Gustavsson.
We apply this to accurate evaluation of the Stokes flow near to a
curved filament in the slender body approximation.
Our 3D method is several times more efficient (both in terms of
kernel evaluations, and in speed in a C implementation) than
the only existing alternative, namely, adaptive integration.
\end{abstract}   

\section{Introduction}

Integral equation methods enable efficient
numerical solutions to piecewise-constant coefficient elliptic boundary
value problems, by representing the solution as a layer potential
generated by a so-called ``density'' function defined on
a lower-dimensional source geometry \cite{atkinson,Kress2014}.
This work is concerned with accurate evaluation of
such layer potentials close to their source,
when this source is a
one-dimensional curve embedded in 2D or 3D space.
In 2D, this is a common occurrence: the curve
is the boundary of the computational domain or an interface
between different materials. In 3D,
such line integrals represent the fluid velocity in
non-local slender-body theory (SBT) for filaments in a viscous (Stokes) flow
\cite{keller76,johnson80,gotz00},
and also may represent the fields in the solution of
electrostatic \cite{laptoroid} or
elecromagnetic \cite{haslamwire} response of thin wire conductors.
Applications of the former
include simulation of straight \cite{Tornberg2006}, flexible
\cite{Tornberg2004,naz17} or closed-loop \cite{Mori2018}
fiber suspensions in fluids.
SBT has recently been placed on a more rigorous
footing as the small-radius
asymptotic limit of a surface boundary integral formulation
\cite{Mori2018,Mori2019,Koens2018}.
In this work we focus on non-oscillatory kernels arising from
Laplace (electrostatic) and Stokes applications,
although
we expect that by singularity splitting (e.g.\ \cite{Helsing2015})
the methods we present could be adapted for oscillatory or Yukawa kernels.

For numerical solution, the density function must first be discretized
\cite{sloan92,Kress2014}.
A variety of methods are used in practice.
If the geometry is smooth enough, and the
interacting objects are far enough away that
the density is also smooth, then a non-adaptive density representation
is sufficient.
Such representations
include low-order splines on uniform grids \cite{Tornberg2004},
global spectral discretization by the periodic trapezoid rule
for closed curves
\cite{Kress2014,Hao2014,Barnett2015},
and, for open curves, global Chebychev
\cite{haslamwire,naz17} or Legendre \cite{Tornberg2006}
expansions.
However, if the density has regions
where refinement is needed
(such as corners or close interactions with other bodies),
a composite (``panel'') scheme is better, as is common with
Galerkin boundary-element methods \cite{sloan92}.
On each panel a high-order representation (such as
the Lagrange basis for a set of Gauss--Legendre nodes)
is used; panels may then be split in either a graded
\cite{sloan92,Helsing2008a,ojalalap} or
adaptive \cite{Ojala2015,Rahimian2016,wu2019} fashion.

Our goal is accurate and efficient evaluation of the potential due
a density given on a panel, at arbitrarily close target points.
Not only is this
crucial for accurate solution evaluation once the density is known,
but is also key to constructing matrix elements for
a Nystr\"om or Galerkin solution \cite{sloan92,Kress2014}
in the case when curves are close to each other \cite{Helsing2008,ojalalap,Barnett2015}.
Specifically,
let the panel $\Gamma$ be an open smooth curve in $\mathbb R^d$, where $d=2$ or
$d=3$. We seek to evaluate the layer potential
\begin{align}
  u(\v x) = \int_\Gamma K(\v x, \v y) f(\v y) \dif s(\v y), \quad \v x \notin \Gamma \subset \mathbb R^d,
  \label{eq:model_laypot}
\end{align}
at the target point $\v x\notin\Gamma$. Here the density $f$
is a smooth function defined on $\Gamma$, and $K$ is a
kernel that has a singularity as $\v x \to \v y$.
Letting $k$ denote some smooth (possibly tensor-valued)
function on $\mathbb R^d$,
in two dimensions the dominant singularity can be either logarithmic,
\begin{align}
  K(\v x, \v y) \sim k(\v x-\v y)\log\abs{\v x - \v y},
\end{align}
or of power-law form
\begin{align}
  K(\v x, \v y) \sim \frac{k(\v x-\v y)}{\abs{\v x - \v y}^{m}}, \quad m=1,2, \dots,
\end{align}
where $m$ is generally even.
In 3D only the latter form arises, for $m$ generally odd.
We will assume that $\Gamma$ is described by the parametrization
$\v g : \mathbb R \to \mathbb R^d$ that maps the standard interval
$[-1,1]$ to $\Gamma$,
\begin{align}
  \Gamma = \left\{ \v g(t) \mid t \in [-1,1] \right\}.
\end{align}
Thus the parametric form of the desired integral \eqref{eq:model_laypot} is
\begin{align}
  u(\v x) = \int_{-1}^1 K\left(\v x, \v g(t)\right) \ft(t) \abs{\v g'(t)} \dif t, \quad \v x \notin \Gamma \subset \mathbb R^d,
  \label{eq:model_laypot_param}
\end{align}
where $\ft(t) := f(\v g(t))$ is the pullback of the density to parameter space.

The difficulty of evaluation of a layer potential with given density
is essentially controlled by the distance of the target point $\v x$ from the curve. 
In particular, if $\v x$ is far from $\Gamma$, the
kernel $K(\v x,\v y)$ varies, as a function of $\v y$,
no more rapidly than the curve geometry itself.
Thus a fixed quadrature rule, once it integrates accurately
the product of density $f$ and any geometry functions
(such as the ``speed'' $\abs{\v g'}$),
is accurate for all such targets.
Here, ``far'' can be quantified
relative to the local quadrature node spacing $h$ and desired accuracy:
e.g.\ in the 2D periodic trapezoid rule setting
the distance should be around $(\log 10)h/2\pi\approx 0.37\,h$ times
the desired number of correct digits (a slight extension of
\cite[Remark~2.6]{Barnett2014}).
However, as $\v x$ approaches $\Gamma$, $K(\v x,\v y)$ becomes increasingly non-smooth
in the region where $\v y$ is near $\v x$,
and any fixed quadrature scheme loses accuracy.
This is often referred to as the layer potential integral
being {\em nearly singular}.

There are several approaches to handling the nearly singular case.
The most crude is to increase the number of nodes used
to discretize the density; this has the obvious disadvantage of
growing the linear system size beyond that needed
to represent the density, wasting computation time and memory.
Slightly less naive is to preserve the original discretization nodes,
but interpolate from these nodes onto a refined quadrature scheme (with higher order 
or
subdivided panels \cite{Tornberg2004}) used only for potential evaluations.
However, to handle arbitrarily close targets
one has to refine the panels in an adaptive
fashion, {\em afresh for each target point}.
As a target approaches $\Gamma$ an increasing level of adaptive refinement is demanded.
The large number of conditionals and interpolations needed in such a scheme
can make it slow.

This has led to more sophisticated ``special'' rules which
do not suffer as the target approaches the curve. They
exploit properties of the PDE, or, in 2D, the relation of
harmonic to analytic functions.
In the 2D high-order setting, several important
contributions have been made in the last decade,
including
panel-based kernel-split quadrature \cite{Helsing2008,Helsing2009,ojalalap,Helsing2015,helsingjiang}, globally
compensated spectral quadrature \cite{Helsing2008,Barnett2015}, quadrature
by expansion (QBX) \cite{Barnett2014,Klockner2013,walaqbx2d},
harmonic density interpolation \cite{perezlap2d3d},
and asymptotic expansion \cite{khatri2d}.
In 3D, little work has been done on high-order accurate evaluation of
nearly singular line integrals. Many of the applications are in the
context of active or passive filaments in viscous fluids, but the
quadrature methods are limited to either using analytical formulae for
straight segments \cite{Tornberg2006}, or regularizing the integrand
using an artificial length scale \cite{Cortez2018,Ho2019,Tornberg2004}.

In this paper, we propose a new panel-based quadrature method for
nearly singular line integrals in both 2D and 3D.
The method incurs no loss of efficiency for targets arbitrarily close
to the curve.
(Note that, in the 2D case it is usual for one-sided limits on $\Gamma$
to exist, but in 3D limits on $\Gamma$ do not generally exist.)
%
Our method builds on the panel-based monomial recurrences
introduced by Helsing--Ojala
\cite{Helsing2008} and extended by Helsing and co-workers
\cite{Helsing2009,helsing_axi,helsing_helm,ojalalap}
\cite[Sec.~6]{helsingjiang}, with a key difference that,
instead of interpolating in the physical plane (which is associated with
the complex plane),
we interpolate in the (real) parametrization variable of the panel.
Rather than exploiting Cauchy's theorem in the spatial complex plane,
the singularity is ``swapped'' by cancellation for one in the
parameter plane; this requires a nearby root-finding search for
the (analytic continuation of the)
function describing the distance between the panel and the target.
It builds on recent methods for computing quadrature error estimates in two
dimensions \cite{AfKlinteberg2018}.
We include a robust and efficient method for such root searches.

This new formulation brings two major advantages:
\begin{enumerate}
\item The rather severe loss of accuracy of the Helsing--Ojala method for curved panels---which we quantify for the first time via classical complex approximation
theory---is much ameliorated.
As we demonstrate, more highly curved panels may be handled with the same number
of nodes, leading to higher efficiency.
\item The method extends to arbitrary smooth curves in 3D. Remarkably,
the search for a root in the complex plane is unchanged from the 2D case.
\end{enumerate}
As a side benefit, our method
automatically generates the information required to determine whether
it is needed, or whether plain quadrature
using the existing density nodes is adequate, using results in
\cite{AfKlinteberg2016quad,AfKlinteberg2018}.

The structure of this paper is as follows.
\Cref{sec:two-dimensions} deals with the 2D
problem: we review the error in the plain (direct) quadrature
(\cref{s:direct}),
review the complex interpolatory method of Helsing--Ojala
(\cref{s:helsing}),
then explain its loss of accuracy for curved panels
(\cref{s:walsh}).
We then introduce and numerically demonstrate
the new ``singularity swap'' method
(\cref{s:real}).
\Cref{sec:three-dimensions} extends the method to
deal with problems in 3D.
\Cref{sec:summary-algorithm} summarizes the entire proposed algorithm,
unifying the 2D and 3D cases.
\Cref{sec:numerical-results} presents further numerical tests
of accuracy in 2D, and accuracy
and efficiency in a 3D Stokes slender body application,
relative to standard adaptive quadrature.
We conclude in \cref{s:conc}.


\section{Two dimensions}
\label{sec:two-dimensions}

Section~\ref{s:direct} reviews the convergence rate for
direct evaluation of the potential, i.e.\ using the native quadrature scheme;
this also serves to introduce complex notation and motivate special quadratures.
Section~\ref{s:helsing}
summarizes the complex monomial interpolatory quadrature of Helsing--Ojala
\cite{Helsing2008,Helsing2009} for evaluation of a given density close to its 2D source panel.
In Section~\ref{s:walsh} we explain and quantify
the loss of accuracy due to panel bending.
Finally in Section~\ref{s:real} we
use the ``singularity swap'' to make a real monomial version that
is less sensitive to bending, and which will form the basis of the 3D method.

\subsection{Summary of the direct (native) evaluation error near a panel}
\label{s:direct}

Here we review known results about the approximation of the integral
\eqref{eq:model_laypot_param}
directly using a fixed high-order quadrature rule.
In the analytic panel case this is best understood by analytic continuation
in the parameter.
Let $t_j$ be the nodes and $w_j$ the weights for an $n$-node
Gauss--Legendre scheme on $[-1,1]$, at which we assume the density is known.%
\footnote{We choose this rule since it is the most common;
however, any other high-order rule with an asymptotic Chebychev
density on $[-1,1]$ would behave similarly.}
Substituting the rule into \eqref{eq:model_laypot_param} gives
the direct approximation
\be
u(\v x) \approx \sum_{j=1}^n W_j K(\v x, \v y_j) f_j~,
\label{direct}
\ee
where the nodes are $\v y_j = \v g(t_j)$, the density samples
$f_j = f(\v y_j)$,
and the modified weights $W_j = w_j|\v g'(t_j)|$.
%
%
For any integrand $F(t)$ analytic in a complex neighborhood of $[-1,1]$,
the error
\be
R_n[F] := \sum_{j=1}^n w_j F(t_j) - \int_{-1}^1 F(t) dt
\label{Eng}
\ee
vanishes exponentially with $n$, with a rate which grows
with the size of this neighborhood.
Specifically, define $E_\rho$, the Bernstein $\rho$-ellipse, as the image
of the disc $|z|<\rho$ under the Joukowski map
\be
t = \frac{z+z^{-1}}{2}~.
\label{jouk}
\ee
(For example, \cref{f:complex}(a) shows this ellipse for
$\rho\approx 1.397$.)
Then, if $F$ is analytic and bounded in $E_\rho$, there is a constant $C$ depending only on $\rho$ and $F$ such that
\be
|R_n[F]| \;\le\; C \rho^{-2n}, \qquad \mbox{ for all } n=1,2,\ldots
\label{GLexp}
\ee
An explicit bound
\cite[Thm.~19.3]{Trefethen2013}
is $C = [64/15(\rho^2-1)] \sup_{t\in E_\rho} |F(t)|$.

Returning to \eqref{eq:model_laypot_param}, the integrand of interest
is $F(t) = K(\v x,\v g(t))\ft(t)|\v g'(t)|$,
for which we seek the largest $\rho$ such that $F$ is
analytic in $E_\rho$.
Bringing $\v x$ near $\Gamma$
induces a singularity near $[-1,1]$ in the
analytic continuation in $t$ of $K(\v x,\v g(t))$, a claim justified
as follows.
Each of the kernels $K(\v x,\v g(t))$ under consideration has
a singularity when the squared distance
\be
R(t)^2 := |\v x - \v g(t)|^2 = (x_1-g_1(t))^2 + (x_2-g_2(t))^2
\label{babyR2}
\ee
goes to zero, which gives $x_1-g_1(t) = \pm i(x_2-g_2(t))$.
Taking the negative case, the singularity (denoted by $t_0$) thus obeys
\be
x_1 + ix_2 = g_1(t_0) +ig_2(t_0)~,
\label{croot}
\ee
and one may check by Schwarz reflection
that the positive case gives its conjugate $\overline{t_0}$.
This suggests identifying
$\mathbb R^2$ with $\cc$ and introducing
\begin{align}
  \zeta &= x_1 + ix_2~, \qquad\qquad \mbox{(target point)}\\
  \tau &= y_1 + iy_2~, \qquad\qquad\, \mbox{(source point on $\Gamma$)}\\
  \gamma(t) &= g_1(t) + i g_2(t)~, \qquad \mbox{(complex parameterization of $\Gamma$)}.
  \label{complexi}
\end{align}
Thus, in this complex notation, the equation \eqref{croot} for $t_0$
is written
\be
\zeta = \gamma(t_0)~.
\label{crootc}
\ee
For now we assume that $\gamma:\cc \to \cc$ is analytic in a sufficiently
large complex neighborhood of $[-1,1]$,
so the nearest singularity is $t_0=\gamma^{-1}(\zeta)$,
the preimage of $\zeta$ under $\gamma$;
see upper and lower panels of \cref{f:complex}(a).
Thus our claim is proved.

Then inverting \eqref{jouk} gives the
elliptical radius $\rho$ of the Bernstein ellipse on which any $t\in\cc$ lies,
\be
  \rho(t) := |t \pm \sqrt{t^2-1}|~, \qquad \mbox{ with sign chosen such that $\rho > 1$~.}
  \label{eq:bernstein_radius}
\ee
We then have, for any of the kernels under study,
that \eqref{GLexp} holds for any
$\rho<\rho(\gamma^{-1}(\zeta))$.
Thus images of the Bernstein ellipses, $\gamma(E_\rho)$, control
contours of equal exponential convergence rate, and thus
(assuming that the prefactor $C$ does not vary much with the target point), at
fixed $n$, also the contours of equal error magnitude.
A more detailed analysis for the Laplace double-layer
case in \cite{AfKlinteberg2016quad} \cite[Sec.~3.1]{AfKlinteberg2018}
improves%
\footnote{In that work the tool was instead the asymptotics of the
characteristic
remainder function of Donaldson--Elliott \cite{Donaldson1972},
and $\rho(t_0)$ was left in the form \eqref{eq:bernstein_radius}.}
the rate to $\rho=\rho(\gamma^{-1}(\zeta))$, predicts the constant,
and verifies that the error contours very closely follow this prediction.
Indeed, the lower panel of our
\cref{f:complex}(a) illustrates, for an analytic
choice of density $f$, that the contour of constant error 
(modulo oscillatory ``fingering'' \cite{AfKlinteberg2018})
well matches the curve $\gamma(E_\rho)$ passing through the target $\zeta$.
The top-left plots in
Figures~\ref{fig:compare_2d_16} and \ref{fig:compare_2d_32}
show the shrinkage of these images of Bernstein ellipses as $n$ grows.

For kernels other than the Laplace double-layer, error results are similar
\cite{Klockner2013,AfKlinteberg2016quad}.
This quantified breakdown in accuracy as the target approaches $\Gamma$
motivates the special quadratures to which we now turn.

\subsection{Interpolatory quadrature using a complex monomial basis}
\label{s:helsing}

We now review the special quadrature of Helsing and coworkers
\cite{Helsing2008,Helsing2009,ojalalap}
that approximates \eqref{eq:model_laypot} in the 2D case.
In the complex notation of \eqref{complexi},
given a smooth density defined on $\Gamma$,
for all standard PDEs of interest (Laplace, Helmholtz, Stokes, etc),
the integrals that we need to evaluate can be written as the contour integrals
\begin{align}
  I_L &= I_L(\zeta) := \int_\Gamma f(\tau)\log(\tau-\zeta)\dif\tau,
  \label{eq:complex_L}\\
  I_m &= I_m(\zeta) :=
  \int_\Gamma \frac{f(\tau)}{(\tau-\zeta)^m} \dif\tau, \qquad m=1,2,\dots .
        \label{eq:complex_m}
\end{align}
\Cref{sec:compl-vari-kern} reviews how to rewrite
some boundary integral kernels as such contour integrals; note that
$f$ may include other smooth factors than merely the density.
Other 2nd-order elliptic 2D PDE kernels may be split into terms of the above
form, as pioneered by Helsing and others in the Helmholtz
\cite{helsing_helm}, axisymmetric Helmholtz
\cite{helsing_axi}, elastostatics \cite[Sec.~9]{Helsing2009},
Stokes \cite{Ojala2015}, and (modified)
biharmonic \cite[Sec.~6]{helsingjiang} cases.

An innovation of the method was to
interpolate the density in the complex coordinate $\tau$ rather
than the parameter $t$. Thus,
\begin{align}
  f(\tau) \approx \sum_{k=1}^n c_k \tau^{k-1}~, \qquad \tau \in \Gamma~.
  \label{monom}
\end{align}
This is most numerically stable if for now we assume that $\Gamma$
has been transformed by rotation and scaling to connect the points $\pm1$.
Here the (complex) images of the quadrature nodes are
$\{\tau_j\}_{j=1}^n := \{\gamma(t_j)\}_{j=1}^n$.
The monomial coefficients $c_k$ can be found by collocation at these nodes,
as follows.
Let $\v c$ be the column vector with entries $\{c_k\}$,
$\v f$ be the vector of values $\{f(\tau_j)\}$,
and $A$ be the Vandermonde matrix with entries
$A_{ij}=\tau_i^{j-1}$, for $i,j=1,\dots,n$.
Then enforcing \eqref{monom} at the nodes results in the $n\times n$ linear system
\begin{align}
  A \v c = \v f
  \label{eq:interp_vander}
\end{align}
which can be solved for the vector $\v c$.
Note that working in this {\em monomial} basis, even in the case
of $\tau_j$ on the real axis, is successful
even up to quite high orders, although
this is rarely stated in textbooks.

\begin{rmk}[Backward stability]\label{r:vandcond}  
It is known for $\{\tau_j\}$ being any set of real or complex nodes
that is not close to the roots of unity
that the condition number of the Vandermonde matrix $A$ in \eqref{eq:interp_vander}
must grow at least exponentially in $n$
\cite{gautschi88,pan16}.
However, as pointed out in \cite[App.~A]{Helsing2008},
extreme ill-conditioning in itself introduces no loss of interpolation accuracy,
at least for $n<50$.
Specifically, as long as a vector $\v c$ has
small relative residual $\|A\v c - \v f\|/\|\v f\|$,
then the resulting monomial is close to the data at the nodes.
Assuming weak growth in the Lebesgue constant with $n$,
the 
interpolant is thus uniformly accurate.
Such a small residual norm, if it exists for the given right-hand side,
is found by a solver that is backward stable.
We recommend either partially pivoted Gaussian elimination
(e.g.\ as implemented by {\tt mldivide} in MATLAB)
which is $\mathcal O(n^3)$, or
the Bj\"orck-Pereyra algorithm \cite{Bjorck1970}
which is $\mathcal O(n^2)$.
\end{rmk}    


As with any interpolatory quadrature, once the
coefficients $\{c_k\}$ have been computed, the integrals
\eqref{eq:complex_L} and \eqref{eq:complex_m} can be approximated as
\begin{align}
  I_L &\approx \sum_{k=1}^n c_k q_k(\zeta)~, \label{ILa}\\
  I_m &\approx \sum_{k=1}^n c_k p_k^m(\zeta)~, \label{Ima}
\end{align}
where the values $\{q_k\}$ and $\{p_k^m\}$ are the exact integrals of the monomial densities
\begin{align}
  q_k(\zeta) &= \int_\Gamma \tau^{k-1} \log (\tau - \zeta) \dif\tau~, \qquad k=1,\dots,n
  \label{qk}
  \\
  p_k^m(\zeta) &= \int_\Gamma \frac{\tau^{k-1}}{(\tau - \zeta)^m} \dif\tau~,
  \qquad k=1,\dots,n~.
      \label{pkm}
\end{align}
The benefit of this (somewhat unusual) complex monomial basis is that these
integrals can be efficiently evaluated through {\em recursion formulae},
as follows.

\subsubsection{Exact evaluation of $p_k^m$ and $q_k$ by recurrence}
\label{s:recur}

Recall that $\Gamma$ has been transformed by rotation and scaling to
connect the points $\pm1$.
We first review the Cauchy kernel $m=1$ case
\cite[Sec.~5.1]{Helsing2008}.
The contour integral for $k=1$ exploits independence of the path $\Gamma$:
\be
p_1^1(\zeta) = \int_\Gamma \frac{\dif \tau}{\tau-\zeta} = \log(1-\zeta)- \log(-1-\zeta)
\pm2\pi i \ {\cal N}_\zeta~.
\label{p11}
\ee
Here ${\cal N}_\zeta \in\mathbb{Z}$ is a winding number that, for the standard branch
cut of the logarithm, is zero if $\zeta$ is outside the domain enclosed
by the oriented closed curve given by
$\Gamma$ traversed forwards plus $[-1,1]$ traversed backwards,
and  ${\cal N}_\zeta= +1$ ($-1$) when $\zeta$ is inside a region enclosed counterclockwise
(clockwise) \cite[Sec.~7]{Helsing2008}.
The following 2-term recurrence is derived by
adding and subtracting $\zeta \tau^{k-1}$
from the numerator of the formula \eqref{pkm}:
\begin{align}
  p_{k+1}^1(\zeta) &= \zeta p_k^1(\zeta) + \frac{1-(-1)^k}{k}, \qquad  k>1~. \label{precur1}
\end{align}
One can then recur upwards to get $p^1_m$ for all $k$.
For $m>1$ we get analogously,
\begin{align}
  p_1^m(\zeta) &= \frac{(1-\zeta)^{1-m} - (-1-\zeta)^{1-m}}{1-m}, \label{p1m}\\
  p_{k+1}^m(\zeta) &= \zeta p_k^m(\zeta) + p_k^{m-1}(\zeta), && k>1~. \label{precurm}
\end{align}
Thus once all values for $m-1$ are known, they can be found for $m$ by
upwards recurrence.  Note that in \cite[Sec.~2.2]{Helsing2009}, a different
recursion was used for the case $m=2$.  Finally, recursive computation
of the complex logarithmic kernel follows directly, following
\cite[Sec.~2.3]{Helsing2009},
\begin{align}
  q_k(\zeta) &= \frac{1}{k} \pars{ \log(1-\zeta) - (-1)^k \log(-1-\zeta) - p_{k+1}^1(\zeta) }~.
  \label{qkz}
\end{align}
We emphasize that these simple exact path-independent formulae are enabled by the choice of the complex monomial basis in the coordinate $\tau$.

\begin{rmk}[Transforming for general panel endpoints]    
\label{r:endpoints}
We now review the changes that are introduced
to the answers $I_L$ and $I_m$ when rescaling and rotating
a general panel $\Gamma$ parameterized by $\gamma(t)$, $t\in[-1,1]$,
into the required one which connects $\pm 1$,
and similarly transforming the target.
Define the complex scale factor $s_0:=(\gamma(1)-\gamma(-1))/2$
and origin $\tau_0 = (\gamma(1)+\gamma(-1))/2$.
Then the affine map from a given location $\tau$ to the transformed
location $\tilde\tau$, given by
\be
\tilde\tau = s(\tau) := \frac{\tau -\tau_0}{s_0}~,
\label{endpointmap}
\ee
can be applied to all points on the given panel and to the target
to give a transformed panel $\tilde\Gamma$ and target $\tilde\zeta$.
From this, applying the formulae above one can evaluate $\tilde I_L$ and
$\tilde I_m$.
By inserting the change of variables one then gets the desired potentials
due to the density on $\Gamma$ at target $\zeta$,
$$
I_L = s_0 \tilde I_L + (\log s_0)\int_\Gamma f(\tau) \dif \tau
~, \hspace{10ex} \mbox{ and } \qquad
I_m = s_0^{1-m} \tilde I_m~, \quad m=1,2,\dots
$$
Here a new integral is needed (the total ``charge''), easily approximated
by $\sum_{j=1}^n w_j |\gamma'(t_j)| \, f(\tau_j)$, where
$w_j$ are the Gauss--Legendre weights for $[-1,1]$.
\end{rmk}   

\subsubsection{Adjoint method for weights applying to arbitrary densities}
\label{s:adjoint}

The above procedure computes $I_L$ and $I_m$
of \eqref{eq:complex_L}--\eqref{eq:complex_m} for a specific function
$f$ (such as the density) on $\Gamma$, using its samples $\v f$ on the nodes
to solve first for a coefficient vector $\v c$.
In practice it is more useful instead to compute the weight vectors
$\v \lambda^L := \{\lambda_j^L\}_{j=1}^n$ and
$\v \lambda^m := \{\lambda_j^m\}_{j=1}^n$ which,
for {\em any} vector $\v f$,
approximate $I_L$ and $I_m$ when their inner products are taken against $\v f$.
That is, $I_L \approx \v f^T \v \lambda^L$ and
$I_m \approx \v f^T \v \lambda^m$, where $T$ indicates non-conjugate transpose.
One application is for
filling near-diagonal blocks of the Nystr\"om matrix arising in a boundary integral formulation:
$(\v \lambda^L)^T$ and $(\v \lambda^m)^T$ form row blocks of the matrix.

We review a result presented in \cite[Eq.~(51)]{Helsing2008}
(where a faster numerical method was also given in the $m=1$ case).
Let $\v \lambda$ be either $\v \lambda^L$ or $\v \lambda^m$, let $I$ be
the corresponding integral \eqref{eq:complex_L} or \eqref{eq:complex_m},
and let $\v p$ be the corresponding vector $\{q_k\}$ or $\{p_k^m\}$
as computed as in Section~\ref{s:recur}.
Writing our requirement for $\v \lambda$,
we insert \eqref{eq:interp_vander} to get
\be
I \approx \v f^T \v \lambda = \v c^T A^T \v \lambda = \v c^T \v p~,
\ee
where the last form simply expresses \eqref{ILa} or \eqref{Ima}.
But this last equality must hold for all $\v c \in \mathbb{C}^n$,
thus
\be
A^T \v \lambda = \v p~,
\label{adjsys}
\ee
which is a linear system whose solution is the weight vector $\v \lambda$.
Notice that this is an adjoint method \cite{adjoint}.

\begin{figure}[ht] 
\mbox{
\mbox{}
\hspace{-3ex}
\raisebox{-1in}{
\includegraphics[width=2.1in]{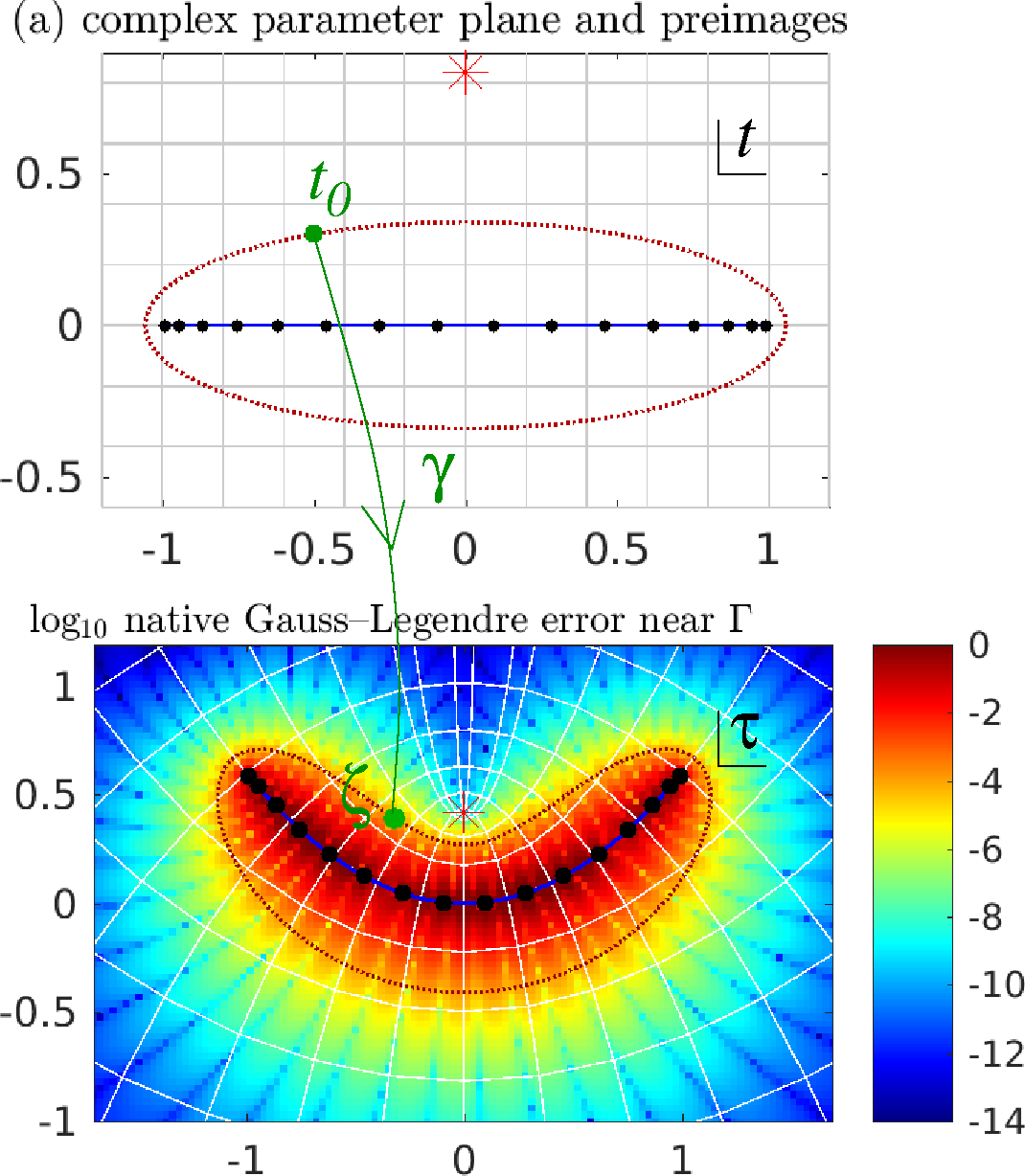}}
\hspace{-2ex}
\raisebox{-.7in}{
\includegraphics[width=1.7in]{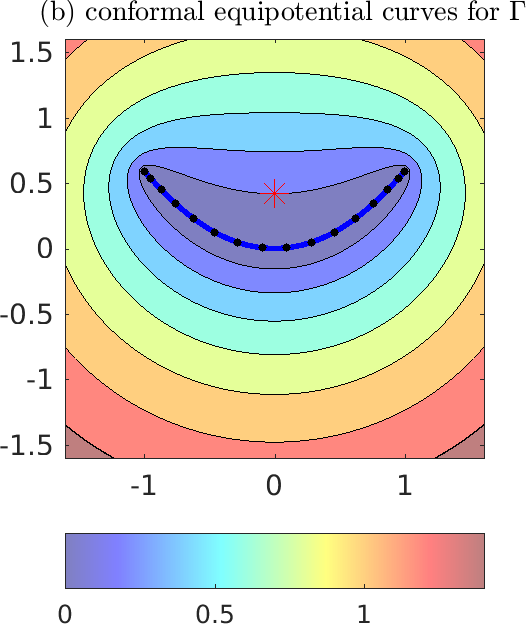}}
\hspace{-2ex}
\raisebox{-.8in}{
\includegraphics[width=2in]{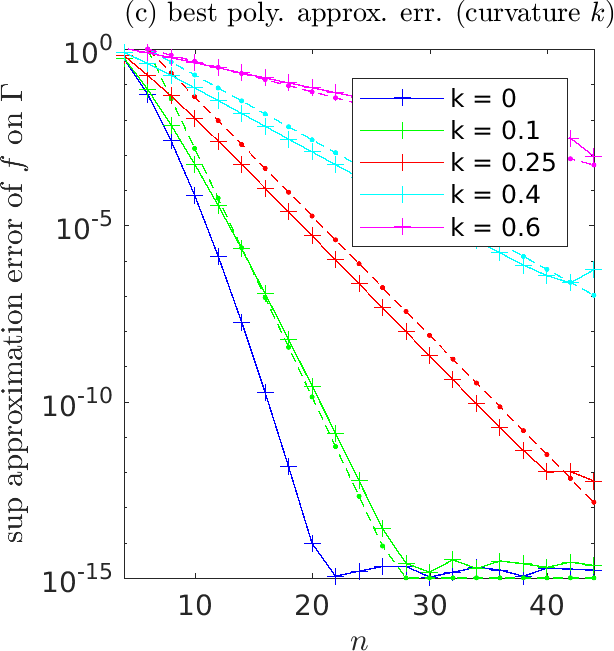}}
}
\caption{Native and complex-interpolatory evaluation errors for a 
simple parabolic panel $\Gamma$, for $I_1$ (complex double-layer) with
density $\ft(t) = \sin(1+3t)$ as a function of parameter $t$.
(a) $t$-plane preimage $[-1,1]$ (top plot) mapping to panel $\Gamma$
(bottom plot)
via $\tau = \gamma(t) = t + ikt^2$ with curvature $k=0.6$, hence minimum
radius of curvature $R=1/2k \approx 0.83$.
Also shown are: $n=16$ Gauss--Legendre points and (bottom with colormap)
their error in potential evaluation using the native (direct) scheme;
an example target $\zeta$ and its preimage $t_0=\gamma^{-1}(\zeta)$ (large green dots);
Bernstein ellipse for $t_0$ and its image (dotted red);
Schwarz singularity $\tau_\ast=i/4k$ and its preimage $t_\ast=i/2k$
(red $\ast$'s); and
Cartesian grid lines (grey) in the $t$-plane and their images (white).
(b) Equipotential curves of $g(z)$ for the same panel $\Gamma$,
showing ``shielding'' of the Schwarz singularity (red $\ast$) thus small
$g\approx0.2$ and $\rho\approx 1.22$.
Contrast $\rho\approx 2.14$ for the Schwarz preimage in the $t$-plane.
(c) Convergence in $n$ of
polynomial approximation error of $f$ over $\Gamma$
(solid lines) for several curvatures $k$, compared to
$C\rho^{-n}$ (dotted) predicted by \cref{s:walsh}.
\label{f:complex}
}
\end{figure}

\subsection{Effect of panel curvature on Helsing--Ojala convergence rate}
\label{s:walsh}

The Helsing--Ojala method reviewed above
is easy to implement, computationally cheap, and can
achieve close to full numerical precision for target points very close to
$\Gamma$. However, there is a severe loss of accuracy as $\Gamma$ becomes curved.
This is striking in the central column of plots in \cref{fig:compare_2d_16}.
This observation led Helsing--Ojala \cite{Helsing2008} and later users
to recommend upsampling from a baseline $n=16$
to a larger $n$, such as $n=32$, {even though the
density and geometry were already accurately resolved} at $n=16$.
The ``HO'' $k=0.4$ plot in \cref{fig:compare_2d_32} shows that
even using $n=32$ has loss of accuracy near a moderately curved panel.
This forces one to split panels beyond what is needed for the density
representation, just to achieve good evaluation accuracy,
increasing the cost.

The error in this method is essentially entirely due to the error of the
polynomial approximation \eqref{monom} of $f$ on $\Gamma$,
because each monomial is integrated exactly (up to rounding error)
as in Section~\ref{s:recur}.
For this error there is a classical result that the best polynomial
approximation converges geometrically in the degree, with rate
controlled by the largest {\em equipotential curve} of $\Gamma$
in which $f$ is analytic.
Namely,
given the set $\Gamma$, let $g$ solve the potential problem in $\mathbb{R}^2$
(which we identify with $\cc$),
\begin{align}
\Delta g &=0\qquad \mbox{ in } \mathbb{R}^2\backslash\Gamma
\nonumber\\
g&=0 \qquad \mbox{ on }\Gamma
\nonumber\\
g(\v x)&\sim \log |\v x| \qquad \mbox{ as } |\v x| \to \infty~.
\nonumber
\end{align}
For each $\rho>1$, define%
\footnote{Note that the reuse of the radius symbol $\rho$ from \cref{s:direct} is deliberate,
since in the special case $\Gamma=[-1,1]$, $C_\rho$ is the
boundary of the Bernstein ellipse $E_\rho$.}
the level curve $C_\rho:=\{\tau\in\cc: g(\tau)=\log\rho\}$.
An example $g$ and its equipotential curves are shown in \cref{f:complex}(b); note that they are very different from the Bernstein ellipse images in the lower
plot of \cref{f:complex}(a).
\begin{thm}[Walsh {\cite[\S4.5]{walsh35}}]  
Let $\Gamma\subset\cc$ be a set whose complement
(including the point at infinity) is simply connected.
Let $f:\Gamma\to\cc$ extend analytically to a function analytic inside and on $C_\rho$, for some $\rho>1$.
Then there is a sequence of polynomials
$p_n(z)$ of degree $n=0,1,\ldots$ such that
$$
|p_n(z) - f(z)| \;\le\; C \rho^{-n}\qquad z\in\Gamma, \mbox{ for all } n=0,1,\ldots
$$
where $C$ is a constant independent of $n$ and $z$.
\end{thm}
%

To apply this to \eqref{monom}, we need to know the largest region in which
$f$ may be continued as an analytic function,
which is in general difficult.
In practical settings the panels will have been refined enough so that
on each panel preimage $\ft$ is interpolated
in the parameter $t\in[-1,1]$ to high accuracy,
thus we may assume that
any singularities of $\ft$ are distant from $[-1,1]$.
Yet, in moving from the $t$ to $\tau$ plane,
the shape of the panel $\Gamma$ itself can introduce
new singularities in $f$ which will
turn out to explain well the observed loss of
accuracy, as follows.
\begin{pro}[geometry-induced singularity in analytic continuation of density]
Let the pullback $\ft$ be a generic function analytic at $t_\ast$.
Let $f(\tau) = \ft(t)$ where $\tau=\gamma(t)$.
Let $t_\ast\in\cc$ be such that $\gamma'(t_\ast)=0$ and $\gamma$
is analytic at $t_\ast$.
Then generically $f(\tau)$ is not analytic at $\tau_\ast = \gamma(t_\ast)$.
\end{pro}
%
\begin{proof}
Since the derivative vanishes, the
Taylor expansion of the parameterization at $t_\ast$ takes the form
$\gamma(t) = \tau_\ast + a_2 (t-t_\ast)^2 + \dots$.
If $f$ were analytic at $\tau_\ast$, expanding gives
$f(\tau) = f_0 + f_1(\tau-\tau_\ast) + \dots$,
then inserting $\tau=\gamma(t)$ gives $f_0 + f_1a_2(t-t_\ast)^2 + \dots$
which must match, term by term, the Taylor expansion of the pullback
$\ft(t)=\ft_0 + \ft_1(t-t_\ast) + \dots$.
But generically $\ft_1 \neq 0$, in which case there is a contradiction.
\end{proof}
Similar (more elaborate) arguments are known for singularities
in the extension of Helmholtz solutions \cite{Mi86}.
Here $\tau_\ast$ is an example of a singularity of the
{\em Schwarz function} \cite{Da74,Shapiro92} for
the curve $\Gamma$.
Recall that the Schwarz function $G$
is defined by $\overline{G(\tau)} = \gamma(\overline{\gamma^{-1}(\tau)})$,
which is the analytic reflection of $\tau$ through the arc $\Gamma$.
At points where $\gamma'=0$, the inverse map $\gamma^{-1}$,
hence the Schwarz function, has a square-root type branch singularity.
\Cref{f:complex}(a) shows that for a moderately bent parabolic panel,
the Schwarz singularity is quite close to the concave side.
Note that the above argument relies on $\ft$ having a generic expansion,
which we believe is typical; we show below that its convergence rate
prediction holds well in practice.

In summary, smooth but bent panels often have a nearby Schwarz singularity,
which generically induces a singularity at this same location
in the analytic continuation of the density $f$; then
the equipotential for $\Gamma$ at this location controls the
best rate of polynomial approximation, placing a fundamental limit on the
Helsing--Ojala convergence rate in $n$.
\Cref{f:complex}(b)
shows that, for a moderately bent panel,
the close singularity is effectively
{\em shielded electrostatically} by the concave panel,
thus the convergence rate $\rho = e^{g(\tau)}$
is very small (close to 1).
In \cref{f:complex}(c) we test this quantitatively:
for a parabolic panel and simple analytic pullback density $\ft$ we
compare the maximum error in $n$-term polynomial approximation on $\Gamma$
with the prediction from the above Schwarz--Walsh analysis.
The agreement in rate is excellent over a range of curvatures.
This also matches the loss of digits
seen in the central column of figures~\ref{fig:compare_2d_16} and \ref{fig:compare_2d_32}:
e.g.\ for even mild curvature $k=0.25$,
$n=16$ achieves only 5 digits, and $n=32$ only 10 digits.


\begin{cnj}[Schwarz singularities at focal points]
Note that in the parabola example of \cref{f:complex},
a square-root type Schwarz singularity appears near
$\Gamma$'s highest curvature point,
at distance $R/2$ on the concave side, where $R$ is the radius of curvature.
This is also approximately true for ellipses, whose foci induce singularities
\cite[\S 9.4]{Shapiro92}.
We also always observe this numerically for other analytic curves.
Hence we conjecture that it is a general result for analytic curves,
in an approximate sense that becomes exact as $R\to 0$.
\end{cnj}

\newlength{\cmpw}
\setlength{\cmpw}{0.23\textwidth}
\newlength{\cmph}
\setlength{\cmph}{0.16\textwidth}
\begin{figure}[!htp]
  \centering
  \includegraphics[width=\cmpw]{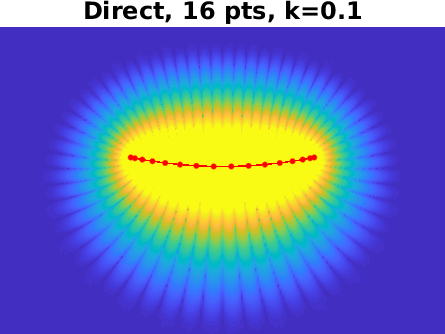} 
  \includegraphics[width=\cmpw]{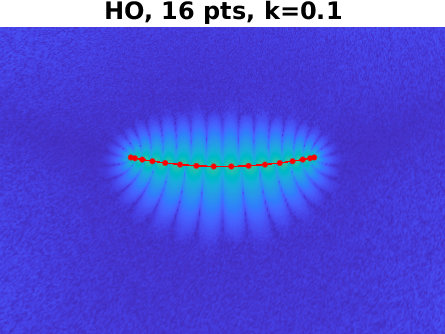}
  \includegraphics[width=\cmpw]{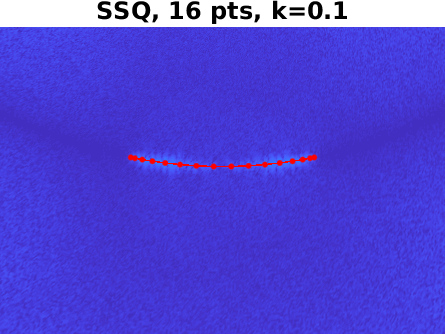}   
  \includegraphics[height=\cmph]{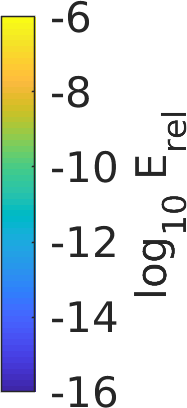}
  \\
  \includegraphics[width=\cmpw]{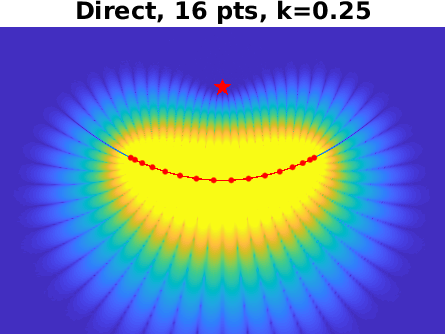}
  \includegraphics[width=\cmpw]{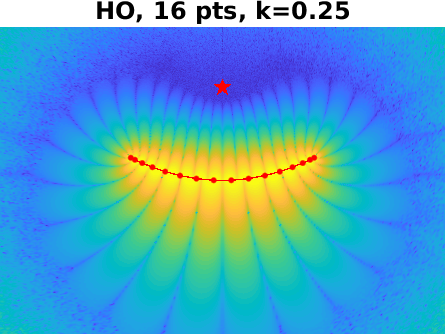}
  \includegraphics[width=\cmpw]{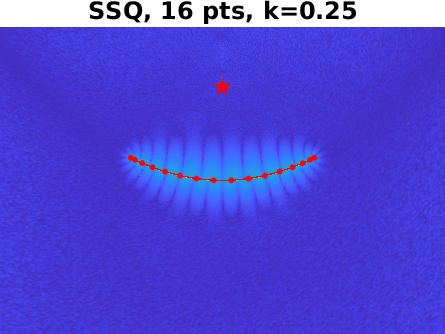}
  \includegraphics[height=\cmph]{colorbar_16_6}
  \\
  \includegraphics[width=\cmpw]{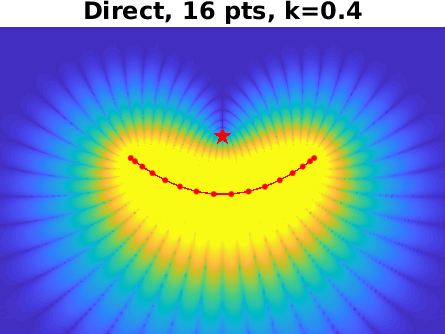}
  \includegraphics[width=\cmpw]{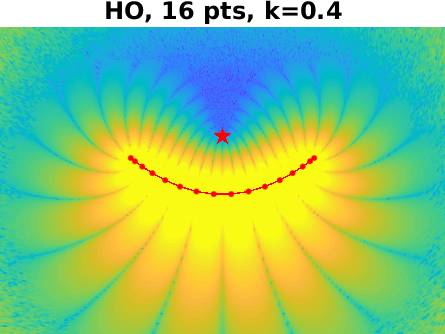}
  \includegraphics[width=\cmpw]{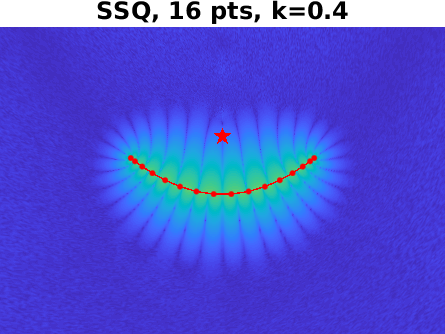}
  \includegraphics[height=\cmph]{colorbar_16_6}
  \caption{Spatial plot of relative errors in 2D quadratures for the
    Laplace double-layer \eqref{eq:lap_dbl_real} on $\Gamma$ a
    parabolic panel $\v g(t) = (t,kt^2)$, $-1\le t \le 1$, with smooth
    density $\rho(\v y) = y_1y_2$, using $n=16$ nodes.  Each row shows
    a different curvature $k$.  Left column: direct quadrature of
    \cref{s:direct}.  Middle column: Helsing--Ojala (HO) quadrature of
    \cref{s:helsing}.  Right column: proposed singularity swap
    quadrature (SSQ) of \cref{s:real}.  The red $\ast$ is the Schwarz
    singularity of $\Gamma$.}
  \label{fig:compare_2d_16}
\end{figure}

\begin{figure}[!hbp]   
  \centering
  \includegraphics[width=\cmpw]{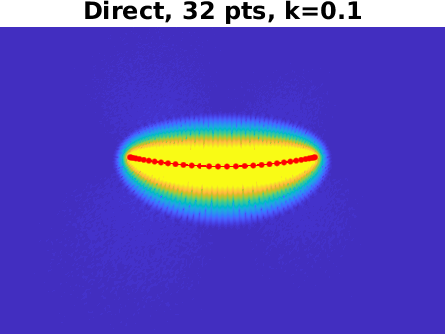}
  \includegraphics[width=\cmpw]{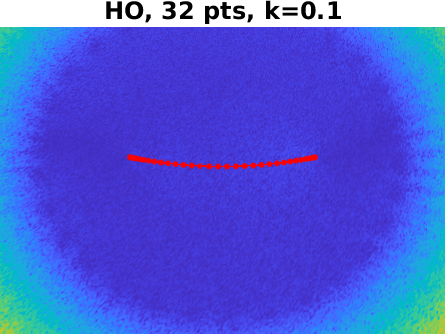}
  \includegraphics[width=\cmpw]{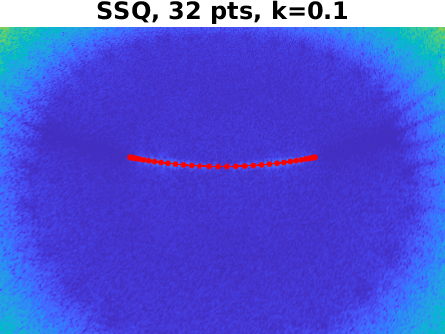}
  \includegraphics[height=\cmph]{colorbar_16_6}
  \\
  \includegraphics[width=\cmpw]{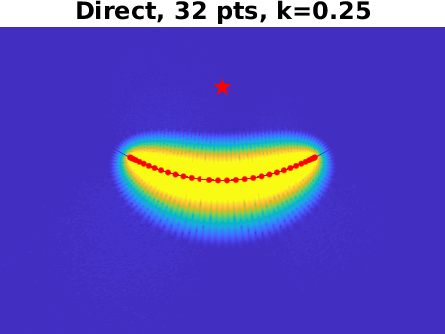}
  \includegraphics[width=\cmpw]{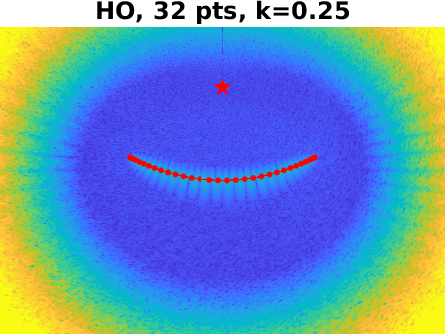}
  \includegraphics[width=\cmpw]{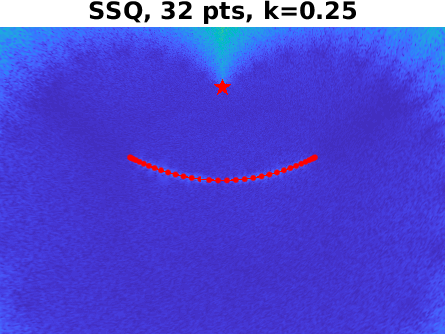}
  \includegraphics[height=\cmph]{colorbar_16_6}
  \\
  \includegraphics[width=\cmpw]{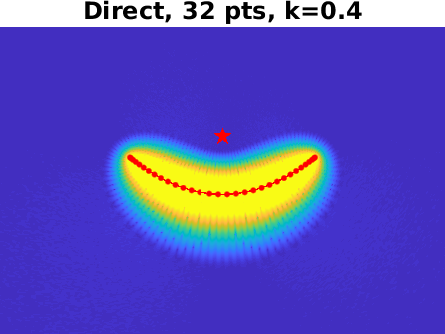}
  \includegraphics[width=\cmpw]{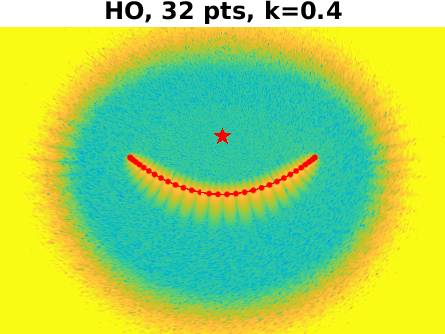}
  \includegraphics[width=\cmpw]{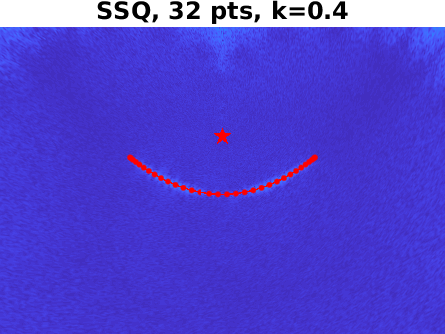}
  \includegraphics[height=\cmph]{colorbar_16_6}
  \caption{Same as \cref{fig:compare_2d_16}, after interpolating panel
    data to $n=32$ nodes.}
  \label{fig:compare_2d_32}
\end{figure}

\subsection{Interpolatory quadrature using a real monomial basis}
\label{s:real}

We now propose a ``singularity swap'' method which pushes the
interpolation problem to the panel preimage $t\in[-1,1]$, thus bypassing the
rather severe effects of curvature just explained.
Recalling that the target point is $\zeta$, we first write the integrals \eqref{eq:complex_L} and
\eqref{eq:complex_m} in parametric form
\begin{align}
  I_L &= \int_{-1}^1 h(t)\log Q(t) \, \dif t~,
        \label{eq:complex_L_param}\\
  I_m &= \int_{-1}^1 \frac{h(t)}{Q(t)^m} \dif t~, \quad m=1,2,\dots,
        \label{eq:complex_m_param}
\end{align}
where we have introduced the displacement function $Q$ and the smooth function $h$,
\begin{align}
  Q(t) &:= \gamma(t)-\zeta,
  \label{eq:Q_def} \\
  h(t) &:= \ft(t) \abs{\gamma'(t)}.
  \label{h_def}
\end{align}
Now let $t_0$ be the root of $Q$ nearest to $[-1,1]$. (Unless the panel
is very curved, there is only one such nearby root and it is simple.)
Then $Q(t)/(t-t_0)$ has a removable singularity at $t_0$, and hence
it and its reciprocal are
analytic in a larger neighborhood of $[-1,1]$ than the above integrands.
By multiplying and dividing by $(t-t_0)$,
\begin{align}
  I_L &= \int_{-1}^1 h(t) \log\frac{Q(t)}{t-t_0} \dif t + 
         \int_{-1}^1 h(t) \log(t-t_0) \dif t~, \label{eq:logint_omega} \\
  I_m &= \int_{-1}^1 \frac{h(t)(t-t_0)^m }{Q(t)^m} \frac{1}{(t-t_0)^m}\dif t~,
  \quad m=1,2,\dots\label{eq:mint_omega}
\end{align}
which as we detail below can be handled as special cases of
\eqref{eq:complex_L} and
\eqref{eq:complex_m} for a new $\Gamma=[-1,1]$ written in the $t$ plane.
We have ``swapped'' the singularity from the $\tau$
to the $t$ plane.
As \cref{f:complex} illustrates, the preimage of the
Schwarz singularity has a much higher conformal distance from $[-1,1]$
than the actual singularity in the $\tau$-plane has from $\Gamma$,
indicating much more rapid convergence.

Thus we now apply the Helsing--Ojala methods of \cref{s:helsing}, but to $[-1,1]$ in the $t$ plane.
Namely, as before, let $t_j$ and $w_j$, $j=1,\dots,n$ be the Gauss--Legendre
quadrature for $[-1,1]$.
The first integral in \eqref{eq:logint_omega} is smooth, so direct quadrature
is used. The second integral in \eqref{eq:logint_omega} is evaluated
via \eqref{ILa} and \eqref{qkz}, setting $\Gamma=[-1,1]$,
with the coefficients $c_k$ for the function $h(t)$ found by
solving the Vandermonde system as in \cref{s:helsing}.
The first term $h(t)(t-t_0)^m/Q(t)^m$ in \eqref{eq:mint_omega} is smooth
and plays the role of $f(\tau)$ in \eqref{eq:complex_m}, so
its coefficients are found similarly. 
\eqref{eq:mint_omega} is then approximated
by \eqref{Ima} with \eqref{p11}--\eqref{precurm}.
\begin{rmk}
Since in this scheme the new flat ``panel'' $[-1,1]$ is fixed,
one may LU decompose the Vandermonde matrix once and for all, then
use forward and back-substitution to solve for the coefficients
$\{c_k\}$ given the $n$ samples of each of the two smooth functions,
namely $\{h(t_j)\}$ and $\{h(t_j)(t_j-t_0)^m/Q(t_j)^m\}$,
with only $\bigO(n^2)$ effort.
\end{rmk}

\begin{rmk}
With a very curved panel it is possible that more than one
root of $Q(t)$ is relevant; see \cref{fig:all_poles}.
In the case where poles $z_1$ and $z_2$ need to be cancelled,
we have derived recursion formulae for $p_k^m(z_1,z_2)$ and $q_k(z_1,z_2)$
which generalize those in \cref{s:recur}.
However, we have not found these to be needed in practice, so omit them.
\end{rmk}

\subsubsection{Finding the roots of $Q(t)$}
\label{sec:finding-poles}

The only missing ingredient in the scheme just described
is to find the nearest complex root $t_0$ satisfying
\begin{align}
  Q(t_0) := \gamma(t_0) - \zeta  = 0~,
\end{align}
i.e.\ the preimage of the (complex) target point $\zeta$ under the
complexification of the panel map $\gamma$; see \cref{fig:all_poles}.
We base our method on that of \cite{AfKlinteberg2018} (where
roots were needed only for the purposes of error
estimation).
Let $\poly_n[\gamma](t)$ denote a degree $n-1$ polynomial approximant
to $\gamma(t)$ on $[-1,1]$.
Using the data at the Legendre nodes $\{t_j\}$, forming
$\poly_n[\gamma](t)$ is both well-conditioned and stable if we use a
Legendre (or Chebyshev) expansion, and has an $\bigO(n^2)$ cost. See \cite{AfKlinteberg2018} for details on how to use a Legendre expansion.
Continuing each basis function to $\mathbb C$
gives a complex approximation
\begin{align}
  \tilde Q(t) := \poly_n[\gamma](t) - \zeta~,
  \label{eq:Qtilde}
\end{align}
for which we seek roots nearest $[-1,1]$.
Given a nearby starting guess, a single root of $\tilde Q$ can be
found using Newton's method, which converges rapidly and costs
$\bigO(n)$ per iteration. A good starting guess is
$t_{\mbox{\tiny init}} \approx s(\zeta)$, recalling that $s$ defined by \eqref{endpointmap}
maps the panel endpoints to $\pm 1$.
With this initialization, Newton iterations converge to the nearest
root in most practical cases. The iterations can, however, be
sensitive to the initial guess when two roots are both relatively
close to $[-1,1]$, and sometimes converge to the second-nearest
root. As illustrated in \cref{fig:all_poles}, this typically happens
only when the target point is on the concave side of a very curved panel.

\begin{figure}  
  \centering
  \hspace{0.01\textwidth}
  \includegraphics[width=0.45\textwidth]{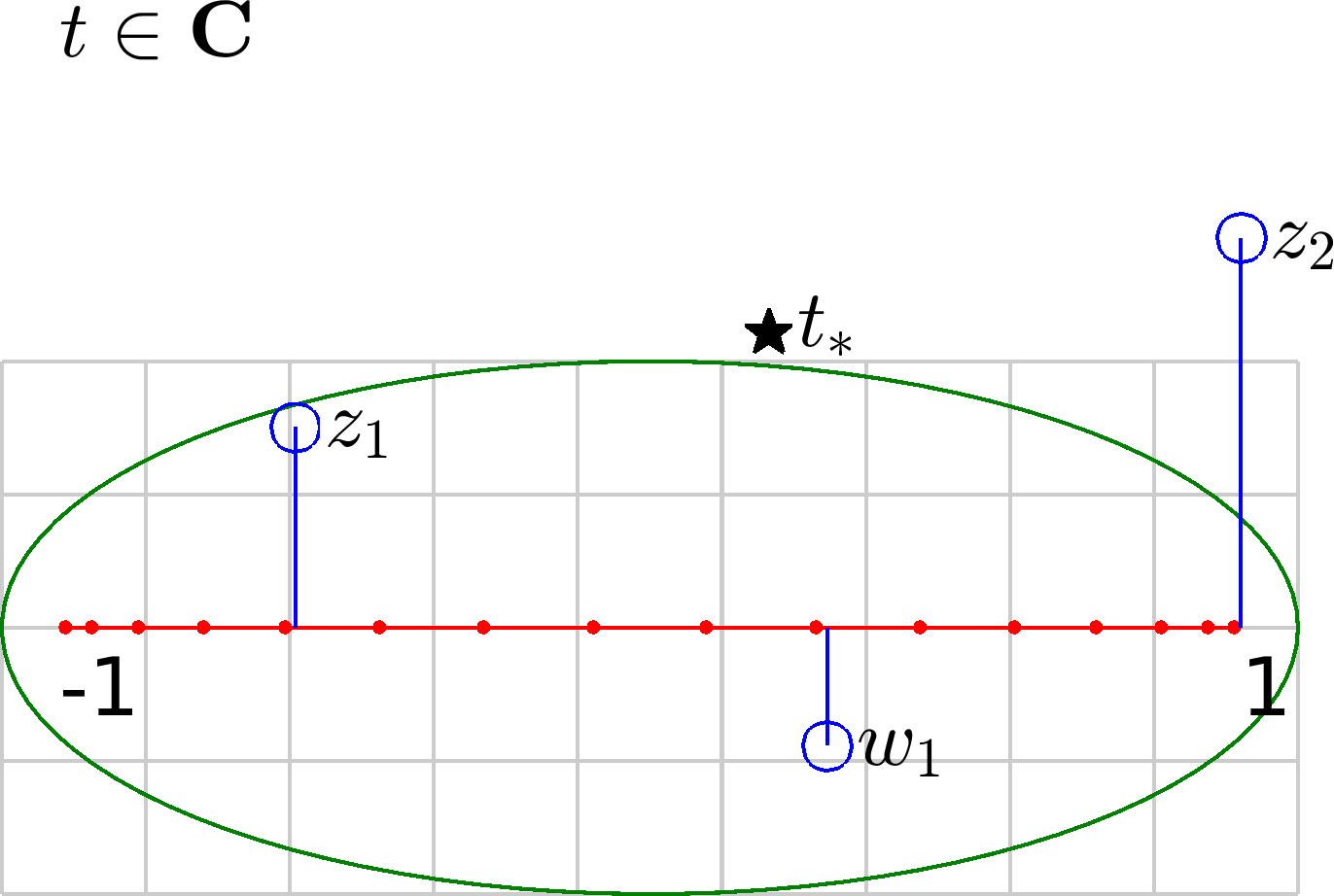}
  \hfill
  \includegraphics[width=0.45\textwidth]{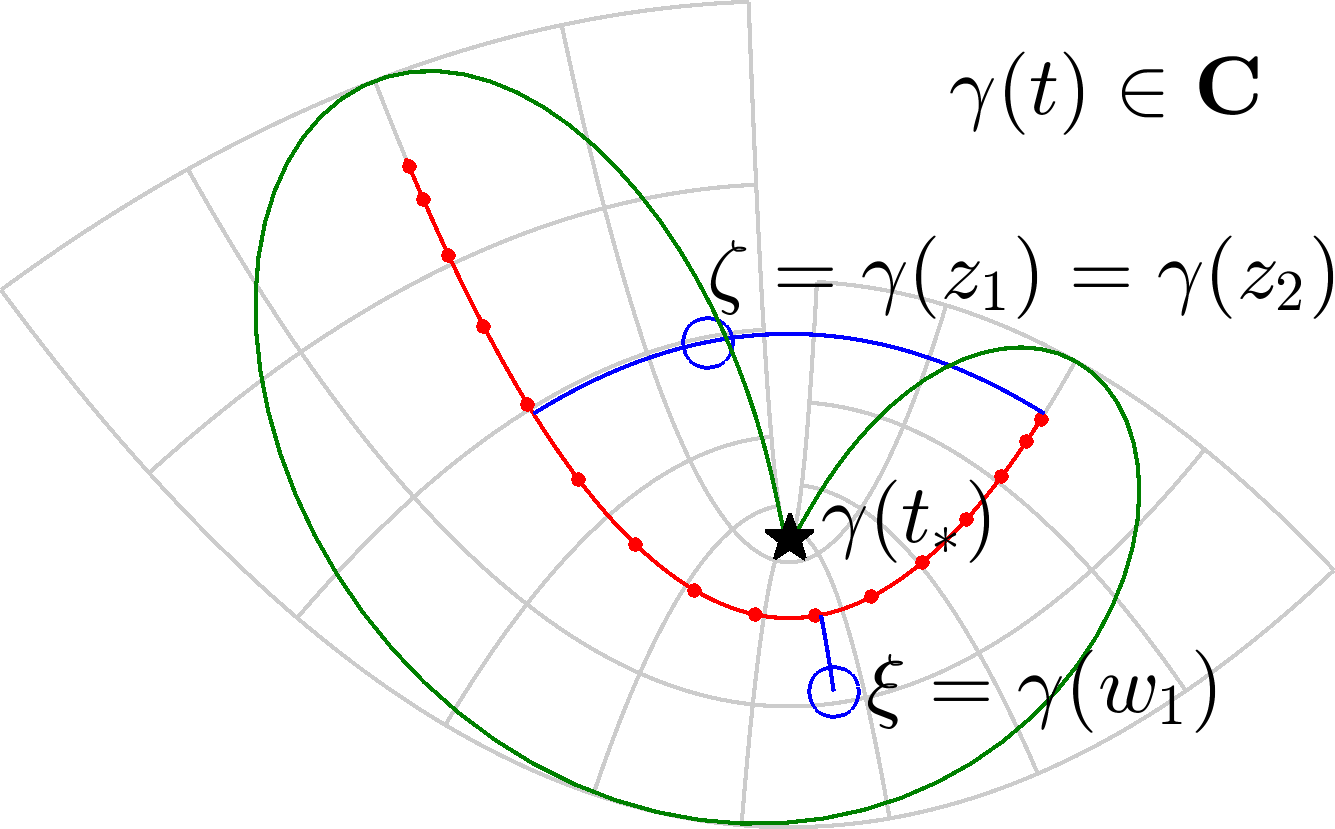}
  \hspace{0.01\textwidth}
  \caption{Illustration of the relationship between the roots of the
    displacement $Q(t)$, to the left, and target points in 2D physical
    space, to the right. The point $\xi$ corresponds to only one
    nearby root, $w_1$, while the point $\zeta$ corresponds to two
    nearby roots, $z_1$ and $z_2$. Note that only $z_1$ lies inside
    the Bernstein ellipse of radius $\Reps$, marked green, such that it requires
    singularity swapping. The point $\gamma(t_{\ast})$ is the Schwarz
    singularity of the curve, marking the point where $\gamma^{-1}$
    stops being single-valued.}
  \label{fig:all_poles}
\end{figure}  

A more robust, but also more expensive, way of finding the nearest
root is to first find all $n-1$ roots of $\tilde Q$ at the same time,
and then pick the nearest one.  This can be done using a matrix-based
method, which finds the roots as the eigenvalues to a generalized
companion (or ``comrade'') matrix \cite{Barnett1975}. This has a cost
that is $\bigO(n^3)$ if done naively, and $\bigO(n^2)$ if done using
methods that exploit the matrix structure \cite{Aurentz2014}.

In order to determine whether Newton iterations or a matrix-based
method should be used for finding the root $t_0$, we suggest a
criterion based on the distance to the nearest Schwarz singularity preimage of
the panel, $t_{\ast}$, as this is a measure of the size of the region
where $\gamma^{-1}$ is single-valued.
Let $\Reps$ be a Bernstein radius beyond which direct Gauss--Legendre
quadrature is expected to have relative error $\epsilon$.
E.g., ignoring the prefactor in \eqref{GLexp} and solving for $\rho$,
\be
\Reps := \epsilon^{-1/2n}
\label{Reps}
\ee
is a useful estimate.
Recalling \eqref{eq:bernstein_radius},
then $\rho(t_{\ast}) \le C \Reps$, where we choose the constant $C=1.1$,
indicates that there may be more
than one nearby root, and that it is safer to use a matrix-based root
finding method. For each panel, $t_{\ast}$ can be found at the time of
discretization, by applying Newton's method to the polynomial
$\poly_n[\gamma'](t)=0$, with the initial guess $t_{\ast} \approx 0$.

\begin{rmk}
\label{r:comrade}
  By switching to matrix-based root finding when there is a nearby
  Schwarz singularity, we get a root finding process that is robust
  also for very curved panels. However, in practical applications we
  have not observed it to be necessary, e.g.\ for the results of
  \cref{sec:numerical-results}. It is on the other hand necessary in
  the examples in \cref{fig:compare_2d_16,fig:compare_2d_32}, because
  we there test the specialized quadrature at distances well {\em beyond}
  where it is required.
\end{rmk}

\begin{rmk}
The root finding process can be made efficient by
computing the interpolants $\poly_n[\gamma](t)$
and $\poly_n[\gamma'](t)$ for each segment at the
time of discretization, such that they can be reused for every target
point $\zeta$.
\end{rmk}

\subsubsection{Tests of improved convergence rate for curved panels}
\label{s:parabola}

We now compare the three methods: i) direct Gauss--Legendre quadrature,
ii) Helsing--Ojala complex interpolatory quadrature,
and iii) the proposed real singularity swap quadrature.
We evaluate the Laplace double layer potential from a single panel,
\begin{align}
  u(\v x) = \int_{\Gamma}
  \rho(\v y)  \frac{(\v y-\v x)\cdot\v n}{\abs{\v y - \v x}^2} \dif s(\v y)~.
  \label{eq:lap_dbl_real}
\end{align}
The panel is the same parabola as in \cref{f:complex},
i.e.\ $\v g(t) = (t, kt^2)$, or in complex form, 
$\gamma(t) = t + ikt^2$,
and we explore the dependence on curvature $k$.
In complex form, which is necessary to apply the
quadratures, the layer potential is (see
\cref{sec:compl-vari-kern})
\begin{align}
 u(\zeta) = 
  - \Im \int_{-1}^1 \frac{\rho(\gamma(t)) \gamma'(t) \dif t }{\gamma(t) - \zeta}~.
  \label{eq:lap_dbl_cpx}
\end{align}

Figs.~\ref{fig:compare_2d_16} and \ref{fig:compare_2d_32} compare
the three schemes for various curvatures,
for $n=16$ and $n=32$ respectively. For $n=32$ all data is upsampled
by Lagrange interpolation from $n=16$ (which is already
adequate to represent it to machine precision).
The error is in all cases measured against a
reference solution computed using adaptive quadrature
(\texttt{integral} in MATLAB with \texttt{abstol} and \texttt{reltol}
set to \texttt{eps}), and the relative error $E_{\text{rel}}$ computed against the maximum value of $u$ on the grid.
To see how they behave when pushed to their
limits, we deliberately apply the quadratures much further away than
necessary, i.e.\ also in the region where direct quadrature achieves
full numerical accuracy. 

The direct errors (left column of each figure) have magnitudes
controlled by the images of Bernstein ellipses under the parameterization,
as explained in \cref{s:direct}. The Helsing--Ojala scheme
(middle column) is much improved over direct quadrature,
but, because of its convergence rates dependence (\cref{s:walsh})
the reduction in accuracy is severe as curvature increases,
although it can to some can extent be ameliorated by using $n=32$.
In addition, there is an error in
the quadrature that grows with the distance from the panel,
particularly for $n=32$, due to amplification of roundoff
error in the upward recurrence for the integrals $\{p_1, \dots, p_n\}$. This
was also noted in \cite{Helsing2008}, where the suggested fix was to
instead run the recurrence backwards for distant points, starting from
a value of $p_n$ computed using the direct quadrature. This avoids the
problem of catastrophic accuracy loss, but only because specialized
quadrature is used where direct quadrature would be sufficient.
The bottom middle plot of \cref{fig:compare_2d_32}
shows at best 10 accurate digits, and that only over quite a narrow range
of distances, making the design of algorithms to give a uniform accuracy
for all target locations rather brittle.

For the proposed singularity swap quadrature (right columns),
the there is still a loss of accuracy with
increased curvature, but the improvement over Helsing--Ojala is striking.
The instability of upwards recurrence is also much more mild.

\begin{rmk} \label{r:upsampling}
  The loss of accuracy with increased curvature can for the
  singularity swap quadrature be explained by considering the roots of
  the displacement function $Q(t)$. We have ``swapped'' out the
  nearest root $t_0$, so the region of analyticity of the regularized
  integrand $h(t)(t-t_0)/Q(t)$ is bounded by the second nearest root
  of $Q(t)$, illustrated by $z_2$ in \cref{fig:all_poles}. This
  limits the convergence rate of the polynomial coefficients
  $\braces{c_k}$, such that we may need to upsample the panel in order
  for the coefficients to fully decay. In the case of our parabolic
  panel, the second nearest root gets closer with increased curvature,
  which explains why upsampling to 32 points is necessary to achieve
  full accuracy in some locations. Note that for the upsampling to be
  beneficial, the coefficients for $k>n$ must be nonzero. This is only
  the case if the components of the integrand ($\rho$, $\gamma$,
  $\gamma'$) are upsampled separately, before the integrand is
  evaluated at the new nodes.
\end{rmk}

%

\section{Line integrals on curved panels in three dimensions}
\label{sec:three-dimensions}

Since the logarithmic kernel is irrelevant in 3D,
we care about the parametric form of \eqref{eq:model_laypot},
\begin{align}
  I_m = I_m(\v x)  = \int_{-1}^1 \frac{\ft(t)}{\abs{\v g(t) - \v x}^m} \abs{\v g'(t)} \dif t ,
  \qquad m=1,3,5,\dots,
\label{eq:3dint_param}
\end{align}
where $\Gamma = \v g([-1,1])$ is an open curve in $\mathbb R^3$,
and $\tilde f$ incorporates the density and possibly other smooth denominator
factors in the kernel $K$.
Fixing the target $\v x$, then
introducing the 3D squared distance function $R^2$ (analogous to \eqref{babyR2})
and the smooth function $h$,
\begin{align}
  R(t)^2 &:= \abs{\v g(t) - \v x}^2 =
           \pars{g_1(t)-x_1}^2 + \pars{g_2(t)-x_2}^2 + \pars{g_3(t)-x_3}^2,
  \label{eq:R2} \\
  h(t) &:= \ft(t) \abs{\v g'(t)}~,
  \label{h_def3d}
\end{align}
we can write \eqref{eq:3dint_param} as
\begin{align}
  I_m = \int_{-1}^1 \frac{h(t)}{ \pars{R(t)^2}^{m/2}} \dif t .
  \label{eq:3dint_R2}  
\end{align}
This integrand has singularities at the roots of $R(t)^2$,
which, since the function is real for real $t$,
come in complex conjugate pairs $\{t_0, \overline{t_0}\}$. How to find
these roots is discussed shortly in \cref{sec:finding-roots-3d}; for now
we assume that they are known.
As in 2D, if the line is not very curved, and the target is nearby,
there is only one nearby pair of roots.

We construct a quadrature for \eqref{eq:3dint_R2}
as in the 2D case \eqref{eq:mint_omega}, except that now there is
a conjugate pair of near singularities to cancel.
Thus we write \eqref{eq:3dint_R2} as
\be
I_m \;=\;
\int_{-1}^1 H(t)
\cdot \frac{1}{\left((t-t_0)(\overline{t-t_0})\right)^{m/2}} \dif t ~,
\qquad \mbox{where }
H(t):=\frac{h(t)\left( (t-t_0)(\overline{t-t_0})\right)^{m/2}}{\pars{R(t)^2}^{m/2}}~.
\label{H}
\ee
$H(t)$ can be expected to be analytic in a much
larger neighborhood of $[-1,1]$ than the integrand.
As in \cref{s:real}, we now represent $H(t)$ in a real monomial basis,
\be
H(t) = \sum_{k=1}^n c_k t^{k-1} ~,
\label{Hmonom}
\ee
getting the coefficients $c_k$ by solving the Vandermonde
system
$$
A \v c = \v h
$$
with the matrix entries $A_{ij} = t_i^{j-1}$, where, as before, $\{t_i\}_{i=1}^n$
are the panel's parameter Legendre nodes in $[-1,1]$. We fill $\v h$
by evaluating its entries $\{H(t_i)\}_{i=1}^n$.
By analogy with \eqref{pkm}, we set
\begin{align}
  P_k^m(t_0) = \int_{-1}^1 \frac{t^{k-1}}{\left((t-t_0)(\overline{t-t_0})\right)^{m/2}} \dif t
  \; = \; \int_{-1}^1 \frac{t^{k-1}}{|t-t_0|^m} \dif t~,
\qquad k=1,\dots, n ~,
  \label{Pkm}
\end{align}
which one can evaluate to machine accuracy
using recurrence relations as described in the next section.
Combining \eqref{Pkm} and \eqref{Hmonom} gives the final
quadrature approximation to \eqref{eq:3dint_param},
\be
I_m \approx \sum_{k=1}^n c_k P_k^m(t_0)~.
\label{Imapprox}
\ee
The adjoint method of \cref{s:adjoint} can again be used
to solve for a weight vector $\v \lambda^m$ whose inner product with
$\v h$ approximates $I_m$.
As in \cref{s:real}, the error in \eqref{Imapprox} is essentially due to
the convergence rate of the best polynomial representation \eqref{Hmonom}
on $[-1,1]$, which is likely to be rapid
because of the absence of curvature effects (\cref{s:walsh}).

\subsection{Recurrence relations for 3D kernels on a straight line}
\label{sec:recurs-form-three}

The above singularity swap has transformed the integral on a
curved line $\Gamma$ in 3D to \eqref{H}, whose denominator
corresponds to that of a {\em straight line} in 3D.
Such upward recursions are available in
\cite[App.~B]{Tornberg2006}, where they were used for Stokes quadratures near
straight segments. We will here present them with some improvements
which increase their stability for $t_0$ in certain regions of $\mathbb C$.
We present the case of a single root pair
$\{t_0,\overline{t_0}\}=t_r \pm it_i$, for orders $m=1,3,5$,
noting that, while we have derived formulae for double root pairs,
we have found that they are never needed in practice.
To simplify notation (matching notation for $I_n^\beta$ in \cite[App.~B]{Tornberg2006}), we write
\begin{align}
  b &:= -2 t_r, & c &:= t_r^2+t_i^2, & d &:= t_i^2 ,
\end{align}
and for the distances to the parameter endpoints we write
\begin{align}
  u_1 &= \sqrt{(1+t_r)^2 + t_i^2} = |1+t_0|, &
  u_2 &= \sqrt{(1-t_r)^2 + t_i^2} = |1-t_0|.  
\end{align}

Beginning with $m=1$ and $k=1$, the integral \eqref{Pkm} is
\begin{align}
  P_1^1(t_0) = \int_{-1}^1 \frac{\dif t}{\sqrt{(t-t_r)^2 + t_i^2}}
  = \log\left( 1-t_r + u_2 \right)
  - \log\left( -1-t_r + u_1 \right) .
  \label{eq:p11_vanilla}
\end{align}
This expression suffers from cancellation for $t_0$ close to $[-1,1]$,
and must be carefully evaluated. To begin with, it is more accurate in
the left half plane, even though the integral is symmetric in $t_r$,
so the accuracy can be increased by evaluating it after the
substitution $t_r \to -|t_r|$. In addition, the argument of the second
logarithm of \eqref{eq:p11_vanilla} suffers from cancellation when
$|t_r| < 1$ and $t_i^2 \ll (1-|t_r|)^2$, even after the
substitution. In this case we evaluate it using a Taylor series in
$t_i$,
\begin{align}
  -(1-|t_r|) + \sqrt{(1-|t_r|)^2 + t_i^2}
  &= (1-|t_r|) \sum_{n=1}^\infty \frac{(-1)^n(2n)!}{(1-2n)(n!)^2(4^n)} \left( \frac{t_i}{1-|t_r|} \right)^{2n}
    =: S_1(t_0).
\end{align}
We achieve sufficient accuracy by applying this series evaluation to
points inside the rhombus described by $4|t_i| < 1-|t_r|$, evaluating
the series to $n=11$. For $k>1$ the upward recursions of \cite{Tornberg2006}
are stable, such that
\begin{align}
  P_1^{1}(t_0)
  &= \log(1+|t_r|+\sqrt{(1+|t_r|)^2 + t_i^2}) \\ & \quad - 
    \begin{cases}
      \log S_1(t_0), & \text{if} \quad 4|t_i| < 1-|t_r|,\\
      \log(-1+|t_r|+\sqrt{(-1+|t_r|)^2 + t_i^2}), & \text{otherwise} ,
    \end{cases} \\
  P_2^{1}(t_0) &= u_2-u_1 - \frac{b}{2} P_1^{1}(t_0), \\
  P_{k+1}^{1}(t_0) &= \frac{1}{k} \pars{
                   u_2-(-1)^{n-1} u_1 + \frac{(1-2k)b}{2} P_{k}^{1}(t_0) - (k-1) c P_{k-1}^{1}(t_0)
                   } .
\end{align}

For $m=3$, the formula from \cite{Tornberg2006}
for $k=1$ contains a conditional
statement for points on the real axis ($d=0$), where the pair of two
conjugate roots $\braces{t_0, \overline{t_0}}$ merge into a double root.  In
finite precision, this property causes a region around the real axis
where the formula is inaccurate, namely two cones 
extending outwards from the endpoints $\pm 1$.
To get high accuracy there, we consider the
integral with shifted limits,
\begin{align}
  P_1^{3}(t_0) = \int_{-1-t_r}^{1-t_r} \pars{s^2+t_i^2}^{-3/2} \dif s
  = \left[ S_3(s) \right]_{-1-t_r}^{1-t_r}~,
\end{align}
where the antiderivative $S_3$ is evaluated by forming a Maclaurin series
of the integrand in $t_i$, and then integrating that series exactly in
$s$, thus
\begin{align}
  S_3(s) 
  = \int \pars{s^2+t_i^2}^{-3/2} \dif s
  \;=\; \frac{\abs{s}}{s^3}\sum_{n=0}^\infty 
  \left(-\frac{1}{4}\right)^{n+1}\!\!
  \frac{ (2 n+2)!}{((n+1)!)^2}
  \pars{\frac{t_i}{s}}^{2n} .
\end{align}
We find that we get sufficiently high accuracy in the region of
interest by truncating this series to 30 terms, denoted
by $\tilde S_3(s)$, and finally evaluating the integrals for $m=3$ as
\begin{align}
  P_1^{3}(t_0) &=
  \begin{cases}
    \tilde S_3(1-t_r) - \tilde S_3(-1-t_r), & \text{if} \quad 0 < \frac{\abs{t_i}}{\abs{t_r}-1} < 0.6, \\
    \frac{1}{2d}\pars{\frac{b+2}{u_2} - \frac{b-2}{u_1}}, & \text{otherwise} ,
  \end{cases}
  \\
  P_2^{3}(t_0) &= \frac{1}{u_1} - \frac{1}{u_2} - \frac{b}{2} P_1^{3}(t_0), \\
  P_{k+1}^{3}(t_0) &=  P_{k-1}^{1}(t_0) - b P_{k}^{3}(t_0) - c P_{k-1}^{3}(t_0).
\end{align}
The series evaluation can be optimized by defining a succession of
narrower cones, using fewer terms in each cone.

For $m=5$ and $k=1$ we need to repeat the process of finding a power
series that we can use in cones around the real axis, extending from
the endpoints. We now consider
\begin{align}
  P_1^{5}(t_0) = \int_{-1-t_r}^{1-t_r} \pars{s^2+t_i^2}^{-5/2} \dif s,
\end{align}
and the antiderivative of the Maclaurin series of the integrand is
\begin{align}
  S_5(s) 
  = \int \pars{s^2+t_i^2}^{-5/2} \dif s
  \;=\;
  \frac{\abs{s}}{s^5}\sum_{n=0}^\infty 
  \frac{(-1)^{n+1} \left(4 n^2+8 n+3\right)}{3 (n+2) 2^{2 n + 1}}
  \frac{ (2n)! }{ (n!)^2  }
  \pars{\frac{t_i}{s}}^{2n} .
\end{align}
Truncating this series to a maximum of 50 terms, denoted by
$\tilde S_5(s)$, we now evaluate the integrals for $m=5$ as
\begin{align}
  P_1^{5}(t_0) &=
  \begin{cases}
    \tilde S_5(1-t_r) - \tilde S_5(-1-t_r), & \text{if} \quad 0 < \frac{\abs{t_i}}{\abs{t_r}-1} < 0.7, \\
    \frac{1}{3d} \pars{ \frac{b+2}{2 u_2^3} - \frac{b-2}{2 u_1^3} + 2 P_1^{3}(t_0) }, & \text{otherwise} ,
  \end{cases}
  \\
  P_2^{5}(t_0) &= \frac{1}{3 u_1^3} - \frac{1}{3 u_2^3} - \frac{b}{2} P_1^{5}(t_0), \\
  P_{k+1}^{5}(t_0) &=  P_{k-1}^{3}(t_0) - b P_{k}^{5}(t_0) - c P_{k-1}^{5}(t_0).
\end{align}
Note that our expression for $P_2^{5}$ is different from the one
found in \cite{Tornberg2006}; ours avoids a conditional.

\subsection{Finding the roots of $R(t)^2$}
\label{sec:finding-roots-3d}

\begin{figure}[t]
  \centering
  \includegraphics[width=0.3\textwidth]{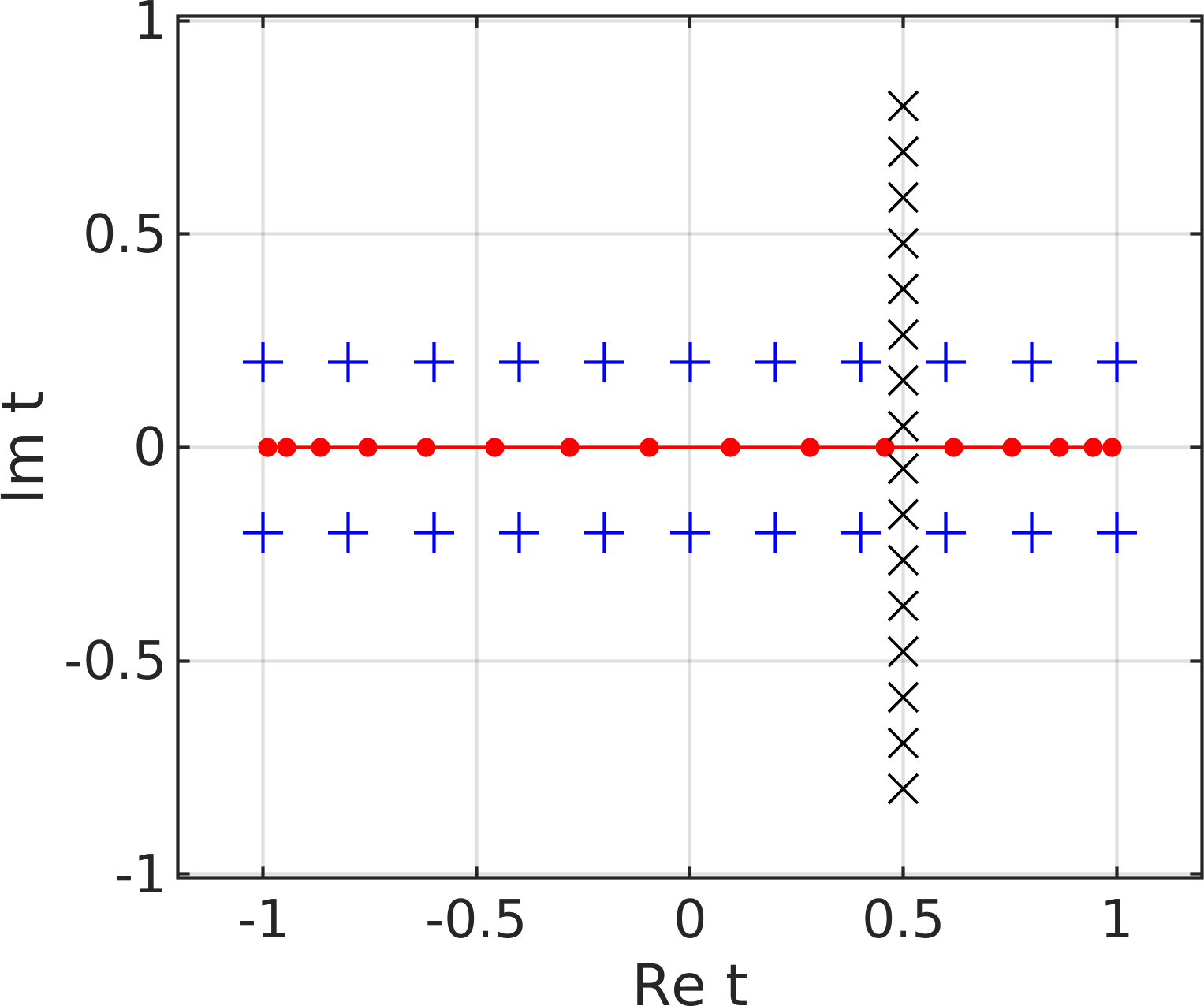}
  \hspace{0.03\textwidth}
  \includegraphics[width=0.3\textwidth]{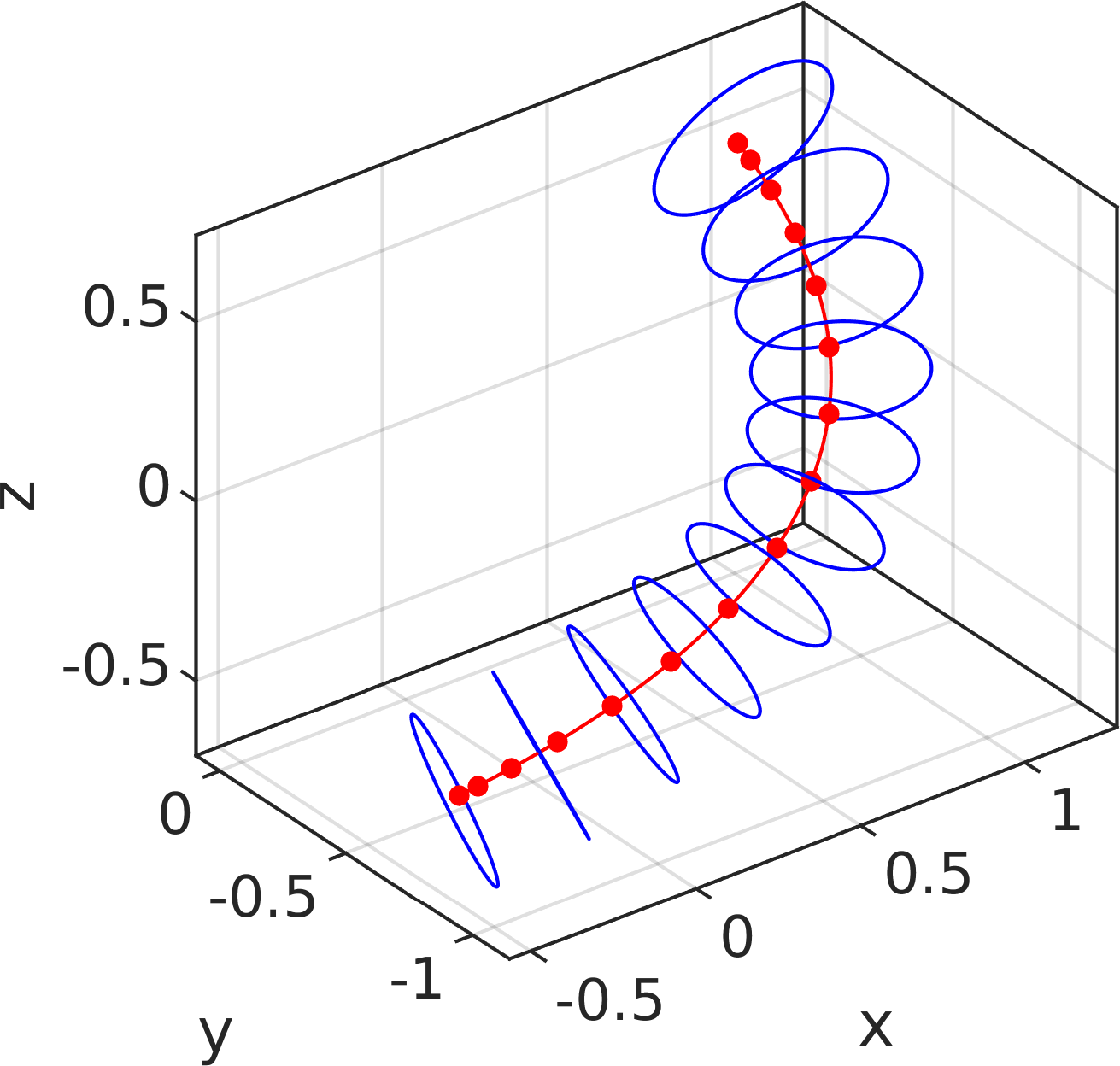}
  \hspace{0.03\textwidth}  
  \includegraphics[width=0.3\textwidth]{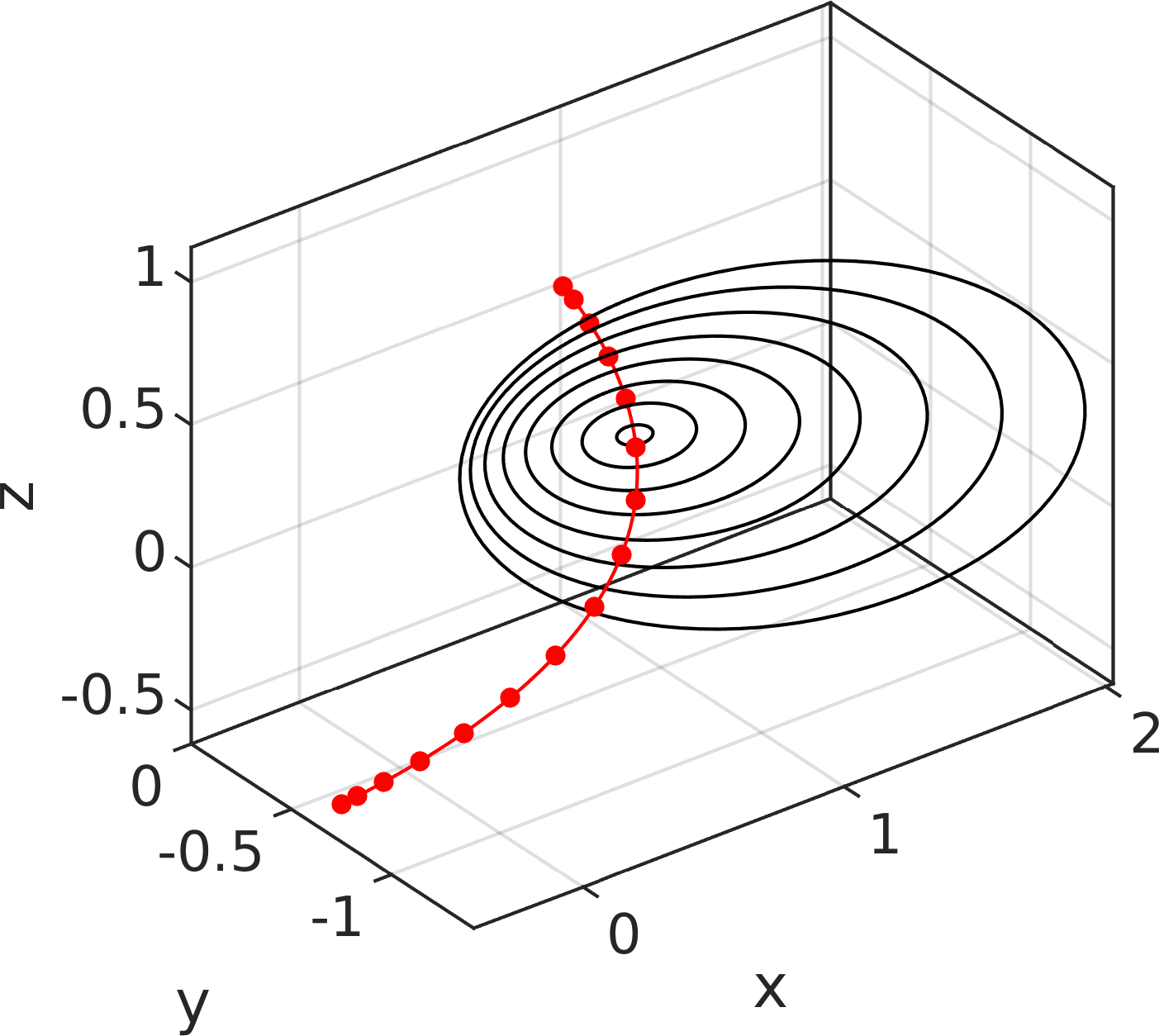}  
  \caption{Correspondence between complex roots of the 3D squared distance
    function $R(t)^2$ in \eqref{eq:R2} and target points in $\mathbb R^3$, close to a curve. 
    The red dots (left) are the Legendre nodes on $[-1,1]$, mapped to a curve in space (center \& right).
    The roots
    marked with a blue \texttt{+} in the left figure correspond to the
    blue circles in the middle figure, while the roots marked with
    a black \texttt{x} in the left figure correspond to the black circles
    in the right figure.}
  \label{fig:poles_3d}
\end{figure}

In three dimensions, the nearby roots of the squared distance function
\eqref{eq:R2} do not correspond to a single point in space, as was the
case in 2D (see \cref{fig:all_poles}). Instead, each complex conjugate
pair of roots $\{t_0, \overline{t_0}\}$ near $[-1,1]$ corresponds to a circle in
space looped around $\Gamma$. To see this, let $\v g_r$ and $\v g_i$ be the
real and imaginary parts of the complexification of the
parametrization,
\begin{align}
  \v g(t) = \v g_r(t) + i \v g_i(t).
\end{align}
Then, the roots of \eqref{eq:R2} satisfy
\begin{align}
  R(t)^2 
  = \abs{\v g_r(t) - \v x} + 2 i \v g_i(t) \cdot \pars{\v g_r(t) - \v x} - \abs{\v g_i(t)}^2
  = 0 .
\end{align}
For this to hold in both real and imaginary components for a given
$t$, the point $\v x$ must satisfy
\begin{align}
  \left\{
  \begin{array}[l]{rcl}    
    \v g_i(t) \cdot (\v g_r(t) - \v x) &=& 0,\\
    \abs{\v g_r(t) - \v x} &=&  \abs{\v g_i(t)}^2.
  \end{array}
  \right .
\end{align}
The above equations describe the intersection of a plane with a
sphere, and are satisfied by the circle of radius $\abs{\v g_i(t)}$
that is centered at $\v g_r(t)$, and lies in the plane normal to
$\v g_i(t)$. All points in $\mathbb R^3$ corresponding to the root $t$
lie on this circle. For an illustration of this, see
\cref{fig:poles_3d}.

In order to construct a quadrature for a given target point $\v x$, we
need to find the roots $t$ to \eqref{eq:R2}, which are the points
satisfying the complex equation
\begin{align}
  R(t)^2 = 
  \pars{g_1(t)-x_1}^2 + \pars{g_2(t)-x_2}^2 + \pars{g_3(t)-x_3}^2 = 0.
\end{align}
We form a
polynomial approximation $\poly_n[\v g](t)$ of $\v g(t)$ using the existing
real Gauss--Legendre nodes $\{\v g(t_j) \}$. We then need to find
the roots of the polynomial
\begin{align}
  \tilde R(t)^2 := 
  \pars{\poly_n[g_1](t)-x_1}^2 +
  \pars{\poly_n[g_2](t)-x_2}^2 +
  \pars{\poly_n[g_3](t)-x_3}^2 .
  \label{eq:R2tilde}
\end{align}
These can be found using either of the root-finding methods discussed in
section \ref{sec:finding-poles}. We here limit ourselves to finding
the root pair $\{t_0,\overline{t_0}\}$
closest to $[-1,1]$, using Newton's method. We find that a
stable initial guess can be obtained using a linear mapping in the
plane containing $\v x$ and the two closest discretization points on
$\Gamma$, which we denote by $\v g_j := \v g(t_j)$ and
$\v g_k :=\v g(t_k)$. The initial guess $t_{\mbox{\tiny init}}$ is then set to one
of the two points satisfying
\be
  \frac{\Re t_{\mbox{\tiny init}}  - t_j}{t_k-t_j} = \frac{
                                        \pars{\v x - \v g_j} \cdot \pars{\v g_k - \v g_j}}
                                        {\abs{\v g_k - \v g_j}^2},
                                        \qquad \mbox{ and } \quad
  \frac{\abs{t_{\mbox{\tiny init}} - t_j}}{\abs{t_k-t_j}}
                                      =
                                        \frac{
                                        \abs{\v x - \v g_j} }{ \abs{\v g_k - \v g_j} }.
\ee
This is the exact root for the case of $\Gamma$ a straight line.

Compared to the 2D case, we find that the accuracy of our method in 3D
is more sensitive to the implementation details of the root-finding
algorithm, especially for roots very close to the interval
$[-1,1]$. Below are a number of noteworthy observations:

\begin{itemize}
  
\item The convergence of Newton's method will deteriorate for roots
  that are close to the real axis (or on it, when $|\Re t_0| > 1$).
  This is because the root $t_0$ and its conjugate $\overline{t_0}$
  will merge into a double root at $\Im t_0=0$, which reduces Newton's
  method to linear convergence. In finite precision, we find that this
  hampers the convergence of Newton's method also for roots that lie
  close the real axis. To ameliorate this problem, we let our
  root-finding routine switch to Muller's method \cite{Press2007} if
  Newton has not converged within a certain number of iterations (we
  use 20). In our tests, this appears to be a stable root-finding
  process for all points requiring special quadrature, such that we
  can apply direct quadrature to points where neither Newton's nor
  Muller's methods converge. For additional robustness, one could
  also switch to the comrade matrix method in such cases.

\item Just as in the 2D case, once the root $t_0$ is known
  we can determine if special quadrature is needed using
  the Bernstein radius $\rho(t_0)$ criterion at the end of
  \cref{sec:finding-poles}.

\item We find that it is important for accuracy to find the roots
  using the approximation \eqref{eq:R2tilde} of $R^2$. If instead
  $R^2$ is approximated as $\poly_n[R^2](t)$, then 1--2 digits of
  accuracy are lost close to $[-1,1]$. If roots are computed from this
  form, for example if the comrade matrix method is used, then they
  can be polished by taking one or two Newton steps of
  \eqref{eq:R2tilde}, thereby recovering the lost accuracy.

\item For panels discretized with more than $n=16$ points, we find
  that the best accuracy is achieved if the roots are found using
  Legendre expansions truncated to 16 terms. In general, our
  recommendation is to not use an underlying discretization with $n$
  larger than 16, and then upsample to more points when evaluating the
  SSQ, if necessary (see \cref{r:upsampling}).
  
\end{itemize}

\section{Summary of algorithm}
\label{sec:summary-algorithm}

Below is a step by step summary of our algorithm for evaluating
the potential generated by a single open curve panel $\Gamma$, 
in 2D or 3D, discretized using $n$ Gauss--Legendre points.
Specifically, we describe how to get the weight vector $\v \lambda$ as
defined in \cref{s:adjoint}, from which one of the line integrals
\eqref{eq:complex_L}, \eqref{eq:complex_m} or \eqref{eq:3dint_param}
can be approximated by
$I= \v \lambda^T \v f$ for any sample vector $\v f$ of
densities (or density
times smooth geometric factors).
Recall that potentials are defined by such line integrals,
possibly in the 2D case via complexification as shown in
Appendix~\ref{sec:compl-vari-kern}.
(We omit the optional finding of the panel's Schwarz singularity,
and use of the robust comrade-matrix method given in \cref{sec:finding-poles},
since we find that it is not needed in practice.)

There is an overall algorithm option (flag) which may take the
values ``no upsampling'', ``upsampling'' or ``upsampling with
upsampled direct''. Either variety of upsampling is more expensive,
but can increase accuracy; the latter avoids some special
weight computations so is the cheaper of the two.

\paragraph{Precomputations for panel $\Gamma$:}
\begin{enumerate}
\item
Choose a suitable polynomial basis (e.g.\ Legendre or
  Chebyshev; we use the former).   
  The analytic continuation of the chosen basis functions must be
  simple to evaluate (as is the case for Legendre).
Form the degree $n-1$
polynomial approximation of the panel parameterization
  $\poly_n[\v g](t)$ in this basis;
this can be done by solving a $n\times n$ linear system 
with right-hand side the
coordinates of the nodes $\v y_j = \v g(t_j)$, $j=1,\dots,n$.
  


\item Given the desired tolerance $\epsilon$, set
the critical Bernstein radius $\rho_\epsilon$ via \eqref{Reps}.
\end{enumerate}

\paragraph{For each given target point $\v x$:}
\begin{enumerate}
\item Use a cheap criterion to check if $\v x$ is a candidate for
  near evaluation, based on the minimum distance between the target
  point and the panel nodes:
  \begin{align}
    \v x \text{ candidate if } \min_{\v y_i \in \Gamma} \abs{\v y_i - \v x} < D,    
  \end{align}
  where $D$ is a multiple of the panel length $h$ (at $n=16$ we use $D=h$ for full accuracy).
  If not a candidate, use the direct rule \eqref{direct} then exit;
  more specifically
  the weights are
  \be
  \lambda_j = w_j |\v g'(t_j)| K(\v x,\v y_j)~,\qquad j=1,\ldots,n.
  \label{directlam}
\ee
\item
Find the complex target preimage $t_0$ from the basis approximation to
$\v g(t)$:
In 2D, use Newton's method applied to \eqref{eq:Qtilde} to find the root $t_0$.
In 3D, use Newton (complemented by Muller's method) applied
to \eqref{eq:R2tilde} to find the root pair $\{t_0,\overline{t_0}\}$.

\item If $\rho(t_0) \ge \Reps$, use the direct rule \eqref{directlam} and exit.
If no upsampling, set $\tilde n=n$, or if upsampling, $\tilde n=2n$.
If ``upsampling with upsampled direct'',
check if  $\sqrt{\Reps} \le \rho(z) < \Reps$,
and if so use the degree $n-1$ approximants to
resample the kernel and density to the $m$ new nodes
and use the upsampled $m$-node version of the direct formula
\eqref{directlam}, then exit.

\item We now do special quadrature for evaluation at $\v x$.
More precisely, we use the singularity swap method
with parameter target $t_0$ and panel $[-1,1]$, as follows.

For 2D, compute the monomial integral vector $\v p$, being
$\{p_k^m(t_0)\}_{k=0}^{\tilde n}$,
or for the logarithmic case $\{q_k(t_0)\}_{k=0}^{\tilde n}$,
via the recurrences \eqref{p11}--\eqref{qkz} applied to $[-1,1]$.
For 3D, it is $\{P_k^m(t_0)\}_{k=0}^{\tilde n}$,
computed via the various recurrences of \cref{sec:recurs-form-three}.

\item 
Solve the adjoint Vandermonde system \eqref{adjsys}
where the $\tilde n$-by-$\tilde n$ matrix $A$ involves the nodes
on $[-1,1]$, and the right-hand side is the $\v p$ just computed.
We find that the Bj\"orck-Pereyra algorithm \cite{Bjorck1970}
is even faster than using a precomputed LU decomposition of $A$.

\item
Finally, one must apply corrections to turn $\v \lambda$ into a set of weights
that act on the samples of $f$, rather than the singularity-swapped
smooth functions involving $h(t)$ or $H(t)$.
First multiply each weight $\lambda_j$ by the Jacobian factor $|\v g'(t_j)|$, according to \eqref{h_def} or \eqref{h_def3d}. Then:

\begin{itemize}
\item For the log case in 2D,
according to \eqref{eq:logint_omega},
add $w_j \log (Q(t_j)/(t_j-t_0))$ to each weight $\lambda_j$.
\item For the power-law case in 2D,
according to \eqref{eq:mint_omega},
multiply each weight by $((t_j-t_0)/Q(t_j))^m$.
\item In 3D,
according to \eqref{H},
multiply each weight by $\bigl( (t_j-t_0)(\overline{t_j-t_0})/R(t_j)^2
\bigr)^{m/2}$.
\end{itemize}

\end{enumerate}

\section{Numerical tests}
\label{sec:numerical-results}

In \cref{sec:two-dimensions} we already performed basic
comparisons of prior methods and the proposed singularity swap quadrature in the 2D case.  Now we give results of more extensive 2D
tests and the 3D case.

The 2D and 3D quadrature methods of this paper have been implemented
in a set of Matlab routines. These are available online
\cite{github-linequad}, together with the programs used for all the
numerical experiments reported here. The computationally intensive
steps of the 3D quadrature have been implemented in C, with basic
OpenMP parallelization, and interfaced to Matlab using MEX. This
allows very rapid computation of quadrature weights: on the computer
used for the numerical results of this section, which is a 3.6 GHz
quad-core Intel i7-7700, we can find the roots (preimages) of
$4.7 \times 10^6$ target points per second, and compute
$3.2 \times 10^6$ sets of target-specific quadrature weights per
second, for $n=16$ (without upsampling).

\subsection{Evaluating a layer potential in two dimensions}
\label{sec:eval-layer-potent}

To study quadrature performance when evaluating an actual layer
potential, we solve a simple test problem using an integral equation
method, then evaluate the solution using 
either the Helsing--Ojala method or our proposed method.
Following \cite{Helsing2008} we test the Laplace equation $\Delta u = 0$ in the
interior of the starfish-shaped domain $\Omega$ given by
$\gamma(t) = (1+0.3\cos(5t)) e^{it}$, $t\in[0, 2\pi)$, and we
choose Dirichlet
boundary conditions $u=u_e$ on $\bdry$ with data
$u_e(\zeta)=\log\abs{3+3i-\zeta}$. We have picked this data to be
very smooth, since the errors then are due to quadrature, rather
than density resolution, making it easier to compare
methods.

We discretize $\bdry$ using composite Gauss--Legendre quadrature, with
$n=16$ points per panel. The panels are formed using the following
adaptive scheme. The interval $[0, 2\pi)$ is recursively bisected into
segments that map to panels, until all panels satisfy a resolution
condition, and all pairs of neighboring segments differ in length by a
ratio this is no more than 2 (i.e.\ level-restriction). The resolution
condition is based on the Legendre expansion coefficients
$\{\hat \gamma'_k\}$ of the panel derivative $\gamma'(t)$: a panel
is deemed resolved to a tolerance $\epsilon$ if
$\max(\abs{\hat \gamma'_{n-1}}, \abs{\hat \gamma'_n}) < \epsilon \norm{\v{\hat \gamma}'}_\infty$.
While somewhat ad-hoc, this produces a solution with a relative error 
that is comparable to $\epsilon$, for a smooth problem  
(see \cref{fig:laplace_lores,fig:laplace_near}).

The solution to the Laplace equation is represented using the double
layer potential \eqref{eq:lap_dbl_cpx}, which results in a second kind
integral equation in $\rho$. This is solved using the Nystr\"om method
and the above discretization, which is straightforward due to the
layer potential being smooth on $\bdry$. For further details on this
solution procedure, see, for example, \cite{Helsing2015} or
\cite{Helsing2008}. Once we have $\rho$, we evaluate the layer
potential using quadrature at sets of points in $\Omega$, and compute
the error by comparing with the exact solution, $u_e$.
We compare two methods:
\begin{enumerate}
\item The Helsing--Ojala (HO) method summarized in \cref{s:helsing}.
\item The proposed singularity swap quadrature (SSQ),
as outlined in \cref{sec:summary-algorithm}.
\end{enumerate}
In both methods we test
``no upsampling'' with $n=16$,
and ``upsampling with upsampled direct'' with $\tilde n=32$.

In our first test, shown in \cref{fig:laplace_lores},
we solve the problem using a quite coarse discretization
(adaptive with $\epsilon=10^{-6}$),
choose $\Reps=1.8$ (corresponding via \eqref{Reps} to a more conservative
error of $10^{-8}$),
and evaluate the 
solution on a $300\times 300$ uniform grid covering the domain.
Far from the boundary the solution shows around the expected 6 digit
accuracy.
Notice that, since only eight panels are used, they are highly
curved.
HO gets only 2 digits at $\tilde n=16$ and scarcely more at $\tilde n=32$,
whereas SSQ gets 3 digits at $\tilde n=16$ and the ``full'' 6 digits at
$\tilde n=32$.
This shows that upsampled SSQ achieves the full expected accuracy
using the efficient discretization chosen for this accuracy, whereas HO
would require further panel subdivision beyond that needed to achieve
the requested Nystr\"om accuracy.

As a second test, we solve the same problem, this time using a fine
discretization (adaptive with $\epsilon=10^{-14}$),
and pick $\Reps=3$,
corresponding via \eqref{Reps} to around $\epsilon\approx 10^{-15}$
at $n=16$.
The result is 32 panels (similar to the 35 panels used in \cite{Helsing2008}).
First we evaluate the solution on a uniform
$250 \times 250$ grid covering all of $\Omega$. Then we evaluate the
solution on the slice of $\Omega$ defined by the mapping $\gamma(t)$
of the square in $\mathbb C$ defined by $\Re t \in [1.66\pi, 1.76\pi]$
and $\Im t \in [d_{\text{min}}, 0.15]$.  Here we use two grids in $t$:
one uniform $250 \times 250$ with $d_{\text{min}}=10^{-3}$, and one
$250\times 250$ that is uniform in $\Re t$ and logarithmic in $\Im t$,
with $d_{\text{min}}=10^{-8}$.
Here we only show the results when using upsampling, since both methods
require that to achieve maximum accuracy
(without upsampling, HO gets 8 digits and SSQ 9 digits).
The results of this test are
shown in \ref{fig:laplace_near}. Both schemes are able to achieve high
accuracy here: both get 13 digits close to the boundary,
and both drop to 11-12 digits when tested extremely close to the boundary.
However,
contrary to the previous test, the proposed SSQ does lose around
1 digit relative to HO, in the distance range $10^{-3}$ to $10^{-8}$.
We do not have an explanation for this loss.

To summarize this experiment, SSQ is much better than HO for
panels that are resolved to lower accuracies, because these
panels may be quite curved.
When panels are resolved to close to machine precision
the methods are about the same, with HO attaining a slightly smaller
error.

\begin{figure}[p]
  \centering
  \includegraphics[width=0.3\textwidth]{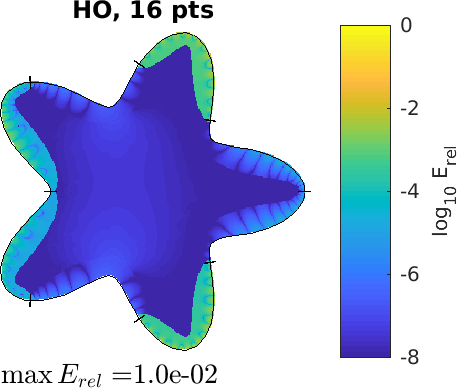}
  \hspace{0.13\textwidth}
  \includegraphics[width=0.3\textwidth]{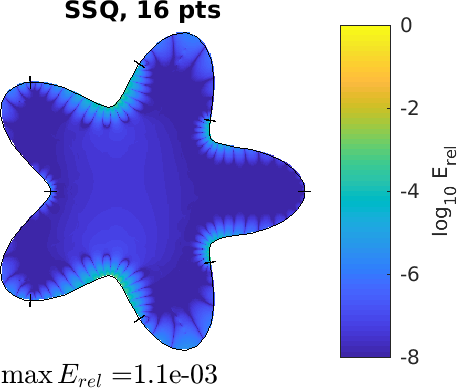}
  \\
  \vspace{0.25cm}
  \includegraphics[width=0.3\textwidth]{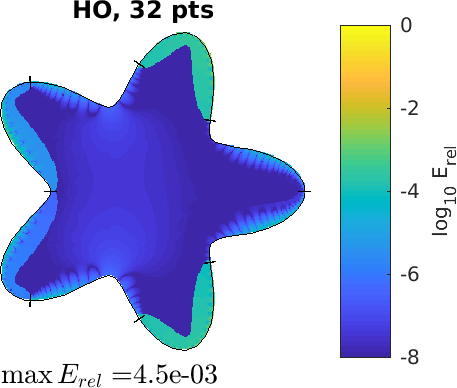}
  \hspace{0.13\textwidth}  
  \includegraphics[width=0.3\textwidth]{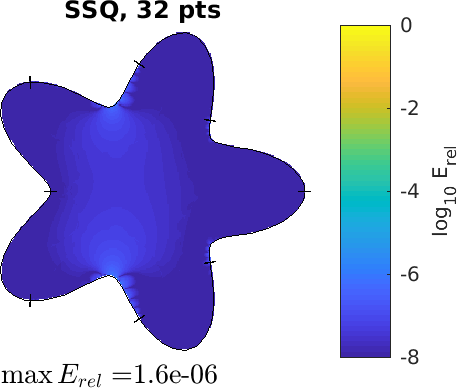}      
  \caption{Comparison of 2D quadratures when evaluating Laplace double
    layer potential using a coarse discretization
    (adaptive 
    with $\epsilon=10^{-6}$, giving 8 panels).
    Left column: Helsing--Ojala (HO). Right column: proposed singularity swap
    quadrature (SSQ).
    Top row: no upsampling. Bottom row: upsampling.}
  \label{fig:laplace_lores}
\end{figure}

\begin{figure}[p]
  \centering
  \includegraphics[height=0.25\textwidth]{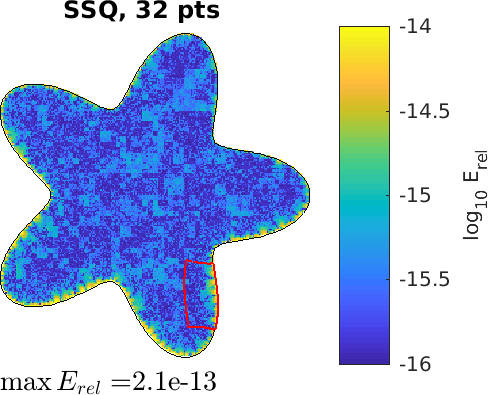}
  \hspace{0.05\textwidth}
  \includegraphics[height=0.25\textwidth]{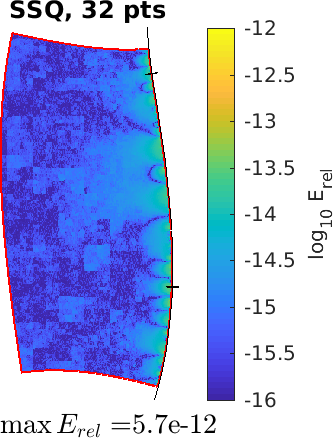}
  \hspace{0.05\textwidth}
  \includegraphics[height=0.25\textwidth]{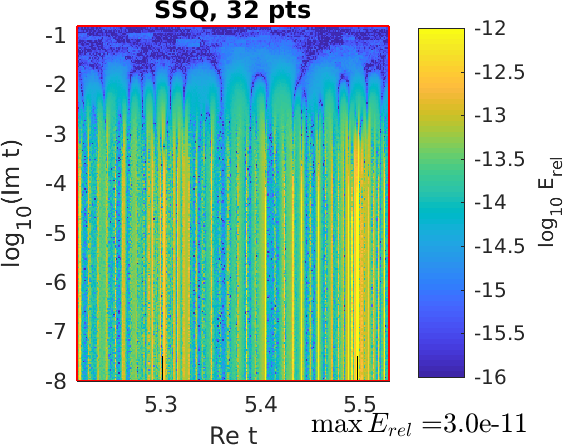}
  \\
  \vspace{0.25cm}
  \includegraphics[height=0.25\textwidth]{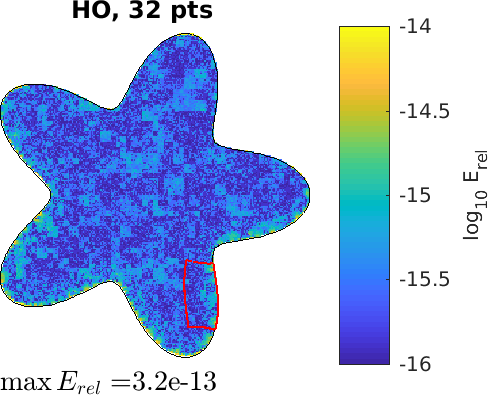}
  \hspace{0.05\textwidth}  
  \includegraphics[height=0.25\textwidth]{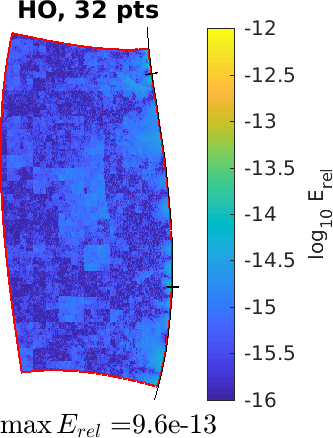}
  \hspace{0.05\textwidth}  
  \includegraphics[height=0.25\textwidth]{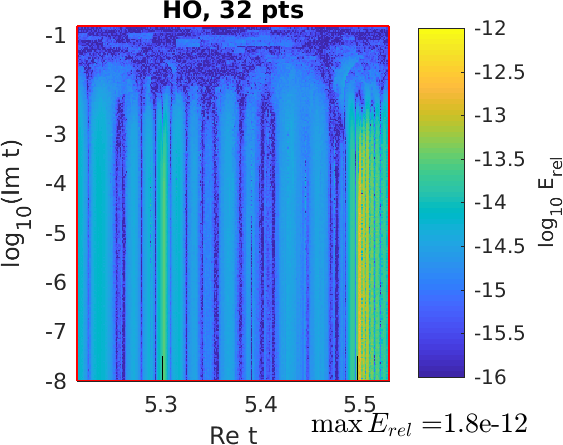}        
  \caption{Comparison of 2D quadratures when evaluating Laplace double
    layer potential using a high-accuracy discretization
    (adaptive 
    with $\epsilon=10^{-14}$, giving 32 panels) and upsampling.
    Top row: proposed singularity swap quadrature (SSQ). Bottom row: Helsing--Ojala (HO).
    Near region (center) is
    marked red in the left plots. Right plot shows error versus root
    location in near region, with exponential grading of the imaginary
    part, proportional to distance. Note the different
    color scale used in the near region.}
  \label{fig:laplace_near}
\end{figure}

\subsection{Evaluating the field due to a slender body in three dimensions}

To test our method in 3D, we evaluate the slender body approximation
of Stokes flow around a thin filament. In this approximation, the flow
due to a fiber with centerline $\Gamma$ and radius $\varepsilon\ll 1$ is given
by the line integral (see e.g. \cite{Mori2018})
\begin{align}
  \v u(\v x) = \int_\Gamma \pars{
  \stokeslet(\v x - \v y) 
  + \frac{\varepsilon^2}{2}\doublet(\v x - \v y) 
  } \v f(\v y) \dif s(\v y)  .
  \label{eq:slender_body_int}
\end{align}
Here $\v f$ is given force data, defined on the centerline, $\stokeslet$
is the Stokeslet kernel, defined as
\begin{align}
  \stokeslet(\v R) = \frac{\v I}{|\v R|} + \frac{\v R \v R^T}{|\v R|^3},
\end{align}
and $\doublet$ is the doublet kernel, defined as
\begin{align}
  \doublet(\v R) = \frac{1}{2}\Delta\stokeslet(\v R)
  = \frac{\v I}{|\v R|^3} - 3\frac{\v R \v R^T}{|\v R|^5} .
\end{align}
Note that $\stokeslet$ and $\doublet$ are tensors, and $\v I$
is the $3\times 3$ identity.
To apply our quadrature, we first write \eqref{eq:slender_body_int} in
kernel-split form,
\begin{align}
  \v u(\v x) &= I_1 + I_3 + I_5,
\end{align}
where
\begin{align}
  I_1 &= \int_\Gamma \frac{\v f(\v y)}{{|\v R|}} \dif s(\v y),  \label{eq:slender_I1}\\
  I_3 &= \int_\Gamma \frac{\pars{\v R \v R^T + \varepsilon^2 \v I/2} \v f(\v y)}{{|\v R|^3}} \dif s(\v y),   \label{eq:slender_I3}\\
  I_5 &= -\frac{3\varepsilon^2}{2} \int_\Gamma \frac{\v R \v R^T \v f(\v y)}{{|\v R|^5}} \dif s(\v y), \label{eq:slender_I5}
\end{align}
and $\v R := \v x - \v y$. Once in this form, one can compute modified
weights for $I_1$, $I_3$, and $I_5$ using the proposed SSQ quadrature.

For our tests, we use the force data $\v f(\v y) = \v y$ and $\varepsilon=10^{-3}$.
We let
$\Gamma$ be the closed curve shown in \cref{fig:long_fiber_field},
which is described by a Fourier series with decaying random
coefficients,
\begin{align}
  \v\gamma(t) &= \Re \sum_{k=-20}^{20} \frac{\v c(k)}{5 + |k|} \hat{\v\gamma}(k) e^{2\pi i k t}, \quad t \in [0, 1),
                \label{eq:squiggle}
\end{align}
where the components of $\v c(k)$ are complex and normally distributed
random numbers\footnote{Matlab: \texttt{rng(0); c =
    (randn(3,41)+1i*randn(3,41));}}.
We subdivide $\Gamma$ into panels by recursively bisecting the curve
until each panel is resolved to a tolerance $\epspan$, using
a criterion similar to that used in 
\cref{sec:eval-layer-potent}:
A panel is deemed resolved if the expansion coefficients $\{\hat s_k\}$
of the speed function $s(t)=\abs{\v\gamma'(t)}$ satisfy
$\max(\abs{\hat s_{n-1}}, \abs{\hat s_n}) < \epsilon \norm{\v{\hat s}}_\infty$.
Each panel is then discretized using
$n=16$ Gauss--Legendre nodes.

\begin{figure}[htbp]
  \centering
  \includegraphics[height=0.25\textwidth]{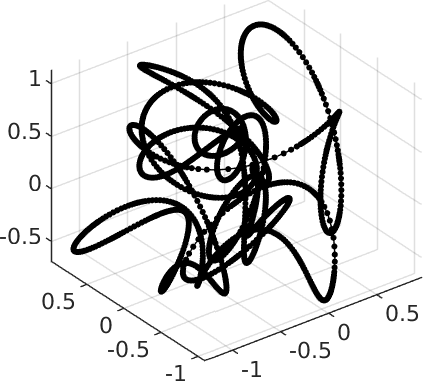}
  \hspace{0.05\textwidth}  
  \includegraphics[height=0.3\textwidth]{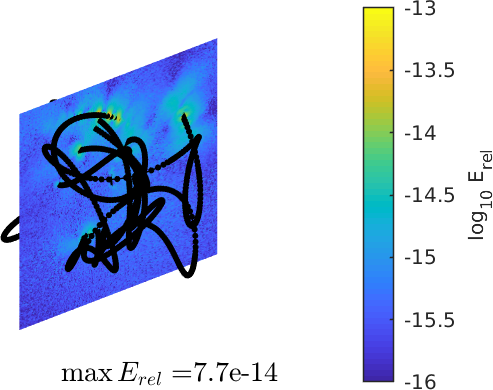}  
  \caption{Left: The curve in 3D described by \eqref{eq:squiggle}, showing
    the discretization nodes (for $\epspan=10^{-10}$) as dots. Right:
    Relative error when evaluating the slender body field
    \eqref{eq:slender_body_int} on the shown slice using
    the proposed SSQ method.}
  \label{fig:long_fiber_field}
\end{figure}

\subsubsection{Comparison to adaptive quadrature}

As a comparison for our method, and also to compute reference values,
we have implemented a scheme for nearly singular 3D line quadrature based on
per-target adaptive refinement: For a given target point $\v x$,
nearby panels are recursively subdivided until each panel $\Gamma_i$
satisfies $\min_{\v y \in \Gamma_i} |\v x - \v y| < h_i$, where $h_i$
is the arc length of $\Gamma_i$. Each newly formed panel is
discretized using $n=16$ Gauss--Legendre points, and the required data
$(\v\gamma, |\v\gamma'|, \v f)$ is interpolated to these points from
the parent panel using barycentric Lagrange interpolation
\cite{Berrut2004}. After this new set of panels is formed, the field
\eqref{eq:slender_body_int} can be accurately evaluated at $\v x$
using the Gauss--Legendre quadrature weights, since all source panels
then are far away from $\v x$, relative to their own length.  This
scheme is relatively simple to implement, but has the drawback of
requiring an increasingly large amount of interpolations and kernel
evaluations as $\v x$ approaches $\Gamma$.

In an attempt to obtain a fair comparison of runtimes between the
adaptive quadrature and the proposed SSQ, we will only report on the
time spent in codes that are of similar efficiency, in terms of
implementation. For the adaptive quadrature, we will therefore report
the time spent on interpolations (reported as $t_{\text{interp}}$),
since those are implemented using a combination of C and BLAS, and
omit the time spent on the recursive subdivisions, since that is
presently implemented in a relatively slow prototype code. For the
SSQ, we will report on the total time spent on finding roots and
computing the new quadrature weights (reported $t_{\text{weights}}$),
since all the expensive steps of that algorithm have been implemented
in C. For both methods we will also report on the time spent on near
evaluations of the Stokeslet and doublet kernels (reported as
$t_{\text{eval}}$), since those are computed using the same code. Far
field evaluation times are omitted, since they are identical.  Thus
$t_{\text{eval}}+t_{\text{interp}}$ is somewhat less than the total
time for an adaptive scheme, whereas
$t_{\text{eval}} + t_{\text{weights}}$ is the total time for our
proposed SSQ.

\subsubsection{Results}

As a first test, we evaluate the $\v u(\v x)$ velocity field
\eqref{eq:slender_body_int} on $200 \times 200$ points covering the
$xz$-slice $[-1.4,1.4] \times 0.25 \times [-1.4,1.4]$. The curve is
adaptively discretized using $\epspan=10^{-10}$, which results in 187 panels. For
each point-panel interaction, we evaluate the SSQ
using data upsampled to 32 Gauss--Legendre points, for all target
points within the Bernstein radius $\Reps=3$. Unlike the 2D case, we find that this
upsampling is necessary for achieving accurate results. The error is
measured against a reference computed using the adaptive quadrature
and a discretization with 18-point panels, created using
$\epspan=5\cdot 10^{-14}$ (this is to ensure that the grids are different, and
that the reference grid is better resolved). The resulting error field, shown in
\ref{fig:long_fiber_field}, indicates that the SSQ can recover at least 13 digits of accuracy at all the
target points. However, the largest errors are at the points
closest to the curve.

To further compare the behavior close to the curve, we compute
$\v u$ at a set of random target points all located at the
same distance $d$ away from $\Gamma$. The test setup and results are
listed in \cref{tab:fiber}. As expected, the computational costs
(kernel evaluations and runtime) of the proposed SSQ quadrature does
not vary with the distance between the target points and the curve, as
long as the target points are within the near evaluation
threshold. The adaptive quadrature, on the other hand, has costs that
grow slowly as the distance decreases, since more and more source
points are required to evaluate the integral. The SSQ also appears to be a factor of several times
cheaper, both in number of evaluations ($4$ to $7$ times less)
and in runtime ($2.5$ to $5$ times faster),
even though our timings are under-estimates for adaptive quadrature.
However, the error for the SSQ in this application
starts to grow for close distances: for instance, at $d=10^{-4}$
the table shows that it loses around 2 digits relative to the
$\epspan$ at which the panels were discretized.

\begin{rmk}
  This loss of accuracy in this application for targets
  at small distance $d$ is not
  entirely surprising. The SSQ method gives weights $\v\lambda$ at
  the given nodes for integrating
the kernels $|\v R|^{-1}$, $|\v R|^{-3}$, and $|\v R|^{-5}$.
The resulting integrals diverge, respectively as
$\log d$, $d^{-2}$, and $d^{-4}$.
We have checked that, for these kernels
with generic densities, the relative errors of SSQ are close
to machine precision (around 13 digits),
even though
as $d\to 0$ the $\lambda_j$ are growing and oscillatory.
However, the Stokes kernels \eqref{eq:slender_I1}--\eqref{eq:slender_I5}
involve near-vanishing numerator factors
(such as $\v R \v R^T$) which partially cancel the singularities,
leading to much smaller values for the integrals.
In our method such factors are incorporated into a
modified density $\tilde f$ in \eqref{eq:3dint_param},
thus catastrophic cancellation as $d\to 0$
is inevitable for this application (this is also true for the
straight fiber method of \cite{Tornberg2006}).
This could only be avoided by a more specialized set of recursions
for the full Stokes kernels \eqref{eq:slender_I3} and
\eqref{eq:slender_I5}.
\label{r:loss}
  \end{rmk}

\begin{table}[t]
\hspace{-2ex}
  \begin{tabular}{|l|l||l|l|l|l||l|l|l|l|l|}
    \hline
    \multicolumn{2}{|c||}{Parameters}
    & \multicolumn{4}{c||}{Adaptive quad.}
    & \multicolumn{4}{c|}{Singularity swap quad. (SSQ)} \\
    \hline
    $d$ & $\epspan$ & $N_{\text{eval}}$ & $t_{\text{eval}}$ & $t_{\text{interp}}$ & $\max E_{\text{rel}}$ & $N_{\text{eval}}$ & $t_{\text{eval}}$ & $t_{\text{weights}}$ & $\max E_{\text{rel}}$ \\
    \hline
    1.0e-02 &  1.0e-10 & 2.8e+06 & 0.17 s & 0.45 s& 7.3e-14 & 6.3e+05 & 0.06 s& 0.19 s& 1.7e-13 \\ 
    1.0e-02 &  1.0e-06 & 4.7e+06 & 0.30 s& 0.78 s& 4.8e-09 & 1.3e+06 & 0.12 s& 0.14 s& 4.8e-09 \\ 
    1.0e-04 &  1.0e-10 & 4.4e+06 & 0.26 s& 0.55 s& 5.9e-11 & 6.3e+05 & 0.06 s& 0.18 s& 2.0e-08 \\ 
    1.0e-04 &  1.0e-06 & 6.3e+06 & 0.38 s& 0.86 s& 5.5e-08 & 1.3e+06 & 0.12 s& 0.13 s& 7.7e-05 \\ 
\hline    
  \end{tabular}
  \caption{Adaptive vs proposed SSQ quadratures for 3D Stokes slender body.
    There are 5000 random target points located
    distance $d$ from the curve $\Gamma$, which is
    discretized using
    tolerance $\epspan$.
For either method,
    the number of near-field kernel evaluations is
    $N_{\text{eval}}$ and the time spent
    on them $t_{\text{eval}}$. $t_{\text{interp}}$ is
    the time interpolating to temporary grids in the adaptive
    quadrature. $t_{\text{weights}}$ is the total time
    finding the interpolatory quadrature weights in the proposed
    SSQ (rootfinding, recursions
    and Vandermonde solve). Error is measured against a reference solution.}
  \label{tab:fiber}
\end{table}

\section{Conclusions}
\label{s:conc}

We have presented an improved version
of the monomial approximation method of Helsing and Ojala \cite{Helsing2008,Helsing2009,ojalalap}
for nearby evaluation of 2D layer potentials discretized with panels.
At the cost of a single Newton search for the target preimage,
our method uses singularity cancellation to ``swap'' the singular problem
on the general curved (complex) panel for one on the standard (real)
panel $[-1,1]$, much improving the error.
Hence we call the method ``singularity swap quadrature'' (SSQ).
We emphasize that it is not the {\em conditioning} of the
monomial interpolation problem that improves---the Vandermonde matrix is exponentially
ill-conditioned in either case---rather,
the exponential {\em convergence rate} improves
markedly in the case of curved panels.
In the case of a Nystr\"om discretization adapted to a requested tolerance,
this allows SSQ to achieve this tolerance when HO would demand further
subdivision.
We gave a quantitative explanation for
both the HO and SSQ convergence rates using
classical polynomial approximation theory:
the HO rate is limited by electrostatic ``shielding'' of a
nearby Schwarz singularity by the panel itself.

We then showed that the singularity swap idea gives a close-evaluation
quadrature for line integrals on general curves in 3D,
which has so far been missing from the literature.
The (to these authors) delightful fact that
complex analysis can help in a 3D application
relies on analytically continuing the squared-distance function
\eqref{eq:R2}.
Note that QBX \cite{Barnett2014,Klockner2013} cannot
apply to line integrals in 3D,
since the potential becomes singular (in the transverse directions)
approaching the curve.
We demonstrated 3D SSQ in a slender body theory Stokes application.
We found that, at least when targets are only moderately close, comparable
accuracies to a standard adaptive quadrature are achieved with several
times fewer kernel evaluations, and, in a simple C implementation,
speeds several times faster.

As \cref{r:loss} discusses, in the Stokes application, in order
to retain full relative accuracy at arbitrarily small target distances
a new set of recurrences for these particular kernels would be essential;
we leave this for future work.
Of course, the slender body approximation breaks down for distances
comparable to or smaller than the body radius, so that the need for
full accuracy at such small distances is debatable.

\section*{Acknowledgments} 

The authors would like to thank Johan Helsing and Charlie Epstein for
fruitful discussions.
L.a.K. would like to thank the Knut and Alice Wallenberg Foundation
for their support under grant no.\ 2016.0410, and the Flatiron Institute for hosting him during part of this work.
The Flatiron Institute is a division of the Simons Foundation.

\appendix
\begin{appendices}

\section{2D kernels in complex variable form}
\label{sec:compl-vari-kern}
\newcommand{\ff}{{\mathfrak f}}

In order to use either the Helsing--Ojala quadrature or the
two-dimensional singularity swap quadrature proposed in this paper,
the complex variable form of the kernel is required.
In \cite[App.~A]{Rahimian2016}) the single and double layer kernels
for several common 2D elliptic PDEs are listed.
If analytically separated into a kernel-split form \cite[Sec.~6]{Helsing2015},
these kernels can be summarized as having singularities of the
following types.
(Note that this is a subset of the forms derived in \cite[Sec.~6]{helsingjiang},
which include denominators up to $|\v y-\v x|^6$.)
\begin{align}
  \int \rho \frac{(\v y-\v x)\cdot\v f}{\abs{\v y - \v x}^2} \dif s(\v y),
  &&&
  \int \rho \log\abs{\v y-\v x} \dif s(\v y),
  \\
  \int \rho \frac{(\v y-\v x)}{\abs{\v y - \v x}^2} \dif s(\v y),
  &&&
  \int \frac{\v f\cdot(\v y-\v x) (\v y-\v x) (\v y-\v x)\cdot\v\nhat}{\abs{\v y - \v x}^4} \dif s(\v y) .
\end{align}
Let us use the following identifications between vectors in
$\mathbb R^2$ and points in $\mathbb C$: $\v x=z$, $\v y=\tau$,
$\v\nhat=\nu$, $\gamma(t)=\v g(t)$, and $\v f=\ff$. Our rewrites use
the following definitions and basic results:
\begin{align}
  \v x \cdot \v y &= \rebrac{z\conj\tau}, \\
  \nu &= i \gamma' / \abs{\gamma'}, \\
  \dif s(\v y) &= i \conj\nu \dif\tau = -i \nu \dif\conj\tau.
\end{align}
Throughout, $\rho$ and $\v f$ are assumed to be arbitrary functions
that are smooth and real-valued. The above integrals can then in complex form be written as
\begin{align}
  \int \rho \frac{(\v y-\v x)\cdot\v f}{\abs{\v y - \v x}^2} \dif s(\v y)
  &= -\Im \int \frac{\rho \ff \conj\nu\dif\tau }{\tau - z},  \label{eq:rdotfker_cpx}\\
  \int \rho \frac{(\v y-\v x)}{\abs{\v y - \v x}^2} \dif s(\v y)
  &= \overline{\int \frac{\rho \conj\nu}{(\tau-z)} i \dif\tau} ,\\
  \int \rho \log\abs{\v y-\v x} \dif s(\v y)
  &= -\Im\int \rho\conj\nu \log(\tau-z) \dif\tau ,\\
  \int \frac{\v f\cdot(\v y-\v x) (\v y-\v x) (\v y-\v x)\cdot\v\nhat}{\abs{\v y - \v x}^4} \dif s(\v y)
  &= \frac{1}{4i}\left[
    \overline{
    \int \frac{ (\conj \tau-\conj z) \ff \dif\tau}{(\tau - z)^2}
    }    
    + \overline{
    \int \frac{(\conj\ff + \ff\conj\nu^2)\dif\tau}{\tau - z}
    }
    - \int \frac{\ff \dif\tau}{\tau-z}
    \right].
\end{align}
As an example of how these can be used, if we let $\v f=\v\nhat$ in
\eqref{eq:rdotfker_cpx} and use that $\nu\conj\nu=1$, we directly get
the complex form of the double layer kernel of the Laplace equation,
\begin{align}
  \int \rho  \frac{(\v y-\v x)\cdot\v n}{\abs{\v y - \v x}^2} \dif s(\v y) = 
  - \Im \int \frac{\rho \dif\tau }{\tau - z}.
\end{align}


\end{appendices}


\bibliography{lib_local,lib_mendeley}
\bibliographystyle{abbrvnat_mod}

\end{document}